\documentclass[a4paper,10pt,reqno
]{amsart}

\usepackage{nomencl}
\usepackage{enumitem}
\usepackage{amssymb,amsmath,amsthm,ascmac,array,bm}
\usepackage[dvipdfmx]{graphicx}
\usepackage{xcolor}

\usepackage{mathtools}

\usepackage{subcaption}
\usepackage{multirow}

\usepackage{cleveref}

\newtheorem{thm}{Theorem}[section]

\newtheorem{lem}[thm]{Lemma}

\newtheorem{rem}[thm]{Remark}
\newtheorem{assump}[thm]{Assumption}

\numberwithin{equation}{section}
\allowdisplaybreaks

\newcommand{\mcl}{\mathcal{L}}

 \newcommand{\lam}{\lambda} \newcommand{\ep}{\epsilon}

\newcommand{\cil}{\xrightarrow{\mcl}} 
\newcommand{\cip}{\xrightarrow{p}} 

\DeclareMathOperator*{\argmin}{arg\,min}

\def\ds#1{\displaystyle{#1}}

\def\nn{\nonumber}

\def\var{{\rm var}}
\def\cov{{\rm cov}}

\def\diag{{\rm diag}}
\def\tr{{\rm tr}}
\def\log{{\rm log}}

\def\sumi{\sum_{i=1}^{n}}

\def\bastinit{\tilde\beta^{\ast(0)}_n}

\def\Zast{Z^\ast}
\def\basttrue{\beta^\ast_0}
\def\gamastn{\Gamma^\ast_n}
\def\bast{\beta^\ast}

\newcommand\smallO{
  \mathchoice
    {{\scriptstyle\mathcal{O}}}
    {{\scriptstyle\mathcal{O}}}
    {{\scriptscriptstyle\mathcal{O}}}
    {\scalebox{.7}{$\scriptscriptstyle\mathcal{O}$}}
  }

\title[Adaptive Ridge Approach to Heteroscedastic Regression]
{Adaptive Ridge Approach to Heteroscedastic Regression}

\author{Ka Long Keith Ho}
\address{Joint Graduate School of Mathematics for Innovation, Kyushu University, 744 Motooka Nishi-ku Fukuoka 819-0395, Japan}
\email{ho.kalongkeith.224@s.kyushu-u.ac.jp}

\author{Hiroki Masuda}
\address{Graduate School of Mathematical Sciences, 
University of Tokyo, 3-8-1 Komaba Meguro-ku Tokyo 153-8914, Japan}
\email{hmasuda@ms.u-tokyo.ac.jp}

\date{\today}

\begin{document}

\title{Adaptive Ridge Approach to Heteroscedastic Regression
}

\begin{abstract}
We propose an adaptive ridge (AR) estimation scheme for a heteroscedastic linear regression model with log-linear noise in data. We simultaneously estimate the mean and variance parameters, demonstrating new asymptotic distributional and tightness properties in a sparse setting. We also show that estimates for zero parameters shrink with more iterations under suitable assumptions for tuning parameters. Aspects of application and possible generalizations are presented through simulations and real data examples.

\end{abstract}

\maketitle

\section{Introduction}\label{sec1}
\label{intro}
Consider the following location-scale regression model
\begin{equation}\label{eq1.1}
    Y_i = X_i^T\alpha + e^{\frac{1}{2}Z_i^T\beta}\epsilon_i, \quad  i=1,...,n  
\end{equation}
for one-dimensional responses $Y_i$, covariates $X_i = (X_{i1},...,X_{ip}) \in \mathbb{R}^p$ and $Z_i = (Z_{i1},...,Z_{iq}) \in \mathbb{R}^q$ with log-linear errors. We focus on the underparameterized scheme with $p,q < n$. The model can also be written as 
\begin{equation}\label{eq1.2}
     Y=X\alpha + D_n(Z;\beta)\epsilon,
\end{equation}
where $X \in \mathbb{R}^{n \times p}$ and $Z \in \mathbb{R}^{n \times q}$ are the design matrices, $Y \in \mathbb{R}^n$ is the response, $\epsilon \in \mathbb{R}^n$ is the error and $D_n(Z;\beta)=D_n(\beta):= \diag(e^{\frac{1}{2}Z_1\beta},...,e^{\frac{1}{2}Z_n\beta})$. In practice, $X$ and $Z$ are often the same or similar, but we will denote them separately for explanatory purposes.

Model \eqref{eq1.1} was first studied by Harvey (1976) \cite{Harvey1976} to study multiplicative heteroscedasticity in linear regression models. The exponential form was chosen because it is simple enough to analyse but also sufficiently expressive for the data-dependent noise structure. Among the various fields where models of this form have been considered, it has been particularly prevalent in modelling volatility in finance. For example, see Engle (1982) and Hsieh (1989) for the renowned ARCH and GARCH models. In this work, we are interested in the estimation of $\theta_0 := (\alpha_0,\beta_0) = (\alpha_{01},...,\alpha_{0p},\beta_{01},...,\beta_{0q}) \in \Theta_\alpha \times \Theta_\beta \subseteq \mathbb{R}^{p}\times \mathbb{R}^{q}$ for bounded domains $\Theta_\alpha$ and $\Theta_\beta$ using the adaptive ridge (AR) scheme, first proposed by Frommlet and Nuel (2016) \cite{PLOSONE}. 

The seminal work of Hoerl and Kennard (1970) \cite{Ridge1970} introduced ridge estimators to tackle multicollinearity, an issue that made the ordinary least squares estimator unreliable. Despite its popularity, ridge regularization is often deemed lacking in its limited ability to identify sparse signals. The adaptive ridge procedure of interest leverages the strengths of $l_2$ regularization while providing much-needed refinements in an underlying sparse model using an iteratively re-weighted ridge penalty. Some of its properties and many applications were discussed in Frommlet and Nuel (2016) \cite{PLOSONE}. Dai et al. (2018) \cite{BAR2018} considered letting the number of iterations go to infinity and named their estimator the Broken Adaptive Ridge (BAR) in a linear model, but under a homogeneous noise setting. Sun et al. (2021) \cite{BARcensored} further refined the discussion of BAR in censored data, and recently Abergel et al. (2024) \cite{ARabergel} provided an extensive overview of the adaptive ridge mechanism and discussed its implementation in depth.

The current work aims to detail the asymptotics of the adaptive ridge under multiplicative heteroscedasticity, and in particular provide desirable theoretical guarantees to the scale parameter that was absent in Harvey (1976) and to the authors' understanding, in subsequent works studying model \eqref{eq1.1}. We also formulate results for general noise distributions, generalizing previous work that was restricted to normal errors. In this work, we consider the asymptotics with the number of iterations fixed and then empirically demonstrate the evolution of the estimates as we iterate. This approach allows for more careful treatment towards small non-zero parameters, see for instance Knight and Fu (2000) \cite[Section 3]{KnightFu2000} and Leeb and P\"otscher (2008) \cite{Leeb2008} about perturbations of signals around 0. This contrasts the work of Dai et al. (2018) \cite{BAR2018} for the BAR, whose theoretical guarantees were based on the existence of the fixed point asymptotically and then showing its oracle properties. Results in this paper therefore greatly complement existing literature as we shed light on the evolution of the iterative scheme by inspecting the asymptotic behaviour at each iterate. Ultimately, we hope to present the AR as a computationally efficient estimation method that can be reliably deployed under multiple statistical settings such as multiplicative heteroscedasticity and sparsity, laying the groundwork for future works to explore more complicated models.

We organize the paper as follows. We introduce notation and the setup of the AR procedure in Section \ref{sec2}. In Section \ref{sec3}, we first state the assumptions, and then prove asymptotic distributional results for the initial and iterated estimators of $\alpha$ and $\beta$ in Theorems \ref{thm1} and \ref{thm2}, respectively. Then, we explain the benefits of performing more iterations in Theorem \ref{thm3}. Simulation results are given in Section \ref{sec4} and we apply the AR on two separate datasets in Section \ref{sec5}. We provide concluding remarks and discuss unresolved issues and future lines of research in Section \ref{sec6} and defer our proofs to the Appendices.

\section{Constructing the Adaptive Ridge Estimators}\label{sec2}
We begin by describing the AR procedure, defining the estimators, and introducing the required notation. For a matrix $A$ and vector $v$, $A^T$ and $v^T$ stand for their transposes and the largest and smallest eigenvalues of $A$ are denoted as $\lam_{\min}(A)$ and $\lam_{\max}(A)$. Unless otherwise specified, $\Vert \cdot \Vert$ is used for the Euclidean norm of vectors and the spectral norm of matrices, so that $\Vert A \Vert = \sqrt{\lam_{\max}(A^TA)}$. If $\lam_{\min}(A) > 0$ ($\geq 0$), we will use $A \succ 0$ ($A \succeq 0$) to mean $A$ is positive definite (semidefinite). We will take exponentials, logarithms, and powers of vectors component-wise and use $\cil$ and $\cip$ for convergence in law and convergence in probability respectively. Finally, $X \lesssim Y$ means $X \leq CY$ for some constant $C$ and $X_n \lesssim Y_n$ means that there exists a universal constant $C$ independent of $n$ so that $X_n \leq CY_n$ for every sufficiently large $n$.

Define the diagonal tuning matrix $\Psi_n := \diag(\psi_{n1},...,\psi_{np}) \succeq 0$. The initial ridge estimator corresponding to $\Psi_n$ is
\begin{equation}\label{eq2.1}
\tilde\alpha^{(0)}_n := \argmin_{\alpha \in \mathbb{R}^p} \left\{\|Y-X\alpha\|^2 + \alpha^T\Psi_n\alpha\right\}=(X^TX+\Psi_n)^{-1}X^TY.
\end{equation}
Next, if we write \[L_n(\alpha)= \left( \log(Y_1-X_1^T\alpha)^2, \dots ,\log(Y_n - X_n^T\alpha)^2\right)^T,\]we get 
\begin{equation}\label{eq2.2}
    L_n(\alpha_0) - Z\beta_0 = \log(\epsilon^2).
\end{equation}
Assuming $\mathbb{P}[\epsilon_i = 0] = 0$ for $i = 1,...,n$ (see Assumption \ref{(A3)}), \eqref{eq2.2} is well-defined with probability one. Subtracting the unknown expectation
\begin{equation}\nn
c_0 := \mathbb{E}\left[\log(\epsilon^2)\right]    
\end{equation}
on both sides then yields
\[
L_n(\alpha_0) - Z\beta_0 - c_0 = \log(\epsilon^2) - c_0,
\]
which enables us to consider a ridge-type estimator for $\beta$ by considering the zero-mean errors $ \log(\epsilon^2) - c_0$. Denote
\begin{center}
    $\Zast  :=  (Z,\mathbf{1}_n)$ and $\basttrue := (\beta_0^T,c_0)^T
    $,
\end{center}
where $\mathbf{1}_n \in \mathbb{R}^n$ denotes a column of 1's. We define the initial ridge estimator of $\basttrue \in \mathbb{R}^{q+1}$ to be
\begin{equation}\label{eq2.3}
\begin{split}
        \tilde\beta^{\ast(0)}_n = \left(\tilde\beta^{(0)}_n,\tilde c^{(0)}_n\right) &\coloneqq \argmin_{\beta^\ast \in \mathbb{R}^{q+1}} \left\{\Vert L_n(\tilde\alpha_n^{(0)})-\Zast \beta^\ast\Vert^2 + \beta^{\ast T}\Omega^\ast_n\beta^\ast\right\} \\
        &= \left(Z^{\ast T} \Zast + \Omega^\ast_n \right)^{-1}Z^{\ast T}L_n(\tilde\alpha_n^{(0)}),
\end{split}
\end{equation}
where $\Omega^\ast_n \coloneqq \diag(\omega_{n1},...,\omega_{nq},\omega_{n,q+1}) \succeq 0$ is the diagonal tuning parameter matrix associated with the initial estimate of $\bast$. We will first show the asymptotic normality of $\tilde\alpha_n^{(0)}$, and then depending on the strictness of conditions placed on $\epsilon$, the tightness or asymptotic normality of $\bastinit$ as well.

\begin{rem}
We assume the matrix $Z$ does not contain a column of $1$'s, which is crucial as we wish to estimate $c_0$, so an additional intercept term poses an identifiability issue. As mentioned in Harvey (1976) \cite{Harvey1976}, if a column of 1's is present, we lose consistency when estimating the intercept term of $\beta$, as there will be bias contributed by the term $c_0$ but the remaining parameters stay unaffected. If the noise structure is assumed, then the expectation $c_0$ can be computed and the problem setup can be simplified.
\end{rem}
Using this estimate for $\beta$, we proceed to introduce an adaptively weighted ridge estimator for $\alpha$. Let \[T(\alpha) \coloneqq \diag(\alpha_1^{-2},...,\alpha_p^{-2}).\]For tuning matrix $\Lambda_n \coloneqq \diag(\lambda_{n1},...,\lambda_{np})$, we define for $k \geq 0$,
\begin{equation}\label{eq2.4}
\begin{split}
    \tilde\alpha_n^{(k+1)} & \coloneqq \argmin_{\alpha \in \mathbb{R}^p} \left \{ \|D_n^{-1}(\tilde\beta_n^{(k)})(Y-X\alpha)\|^2 + \alpha^T\Lambda_n T(\tilde\alpha_n^{(k)})\alpha \right \} \\
 & = \left(X^TD^{-2}_n(\tilde\beta_n^{(k)})X+\Lambda_nT(\tilde\alpha_n^{(k)})\right)^{-1}X^TD^{-2}_n(\tilde\beta_n^{(k)})Y.
\end{split}
\end{equation}
We then repeat a process similar to \eqref{eq2.3} upon obtaining an updated estimator for $\alpha$ with adjusted weights in the ridge penalty. Namely, by setting \[S(\bast) := \diag(\beta_1^{-2},...,\beta_q^{-2},c^{-2}),\] for $k \geq 0$ we define an updated estimator of $\beta^\ast$ as follows:
\begin{equation}\label{eq2.5}
    \begin{split}
        \tilde\beta^{\ast(k+1)}_n &= (\tilde\beta^{(k+1)}_n,\tilde c^{(k+1)}_n)\\
        &\coloneqq \argmin_{\bast \in \mathbb{R}^{q+1}} \left\{\Vert L_n(\tilde\alpha_n^{(k+1)})-\Zast \bast\Vert^2 + \beta^{\ast T}\gamastn S(\tilde\beta_n^{\ast(k)})\bast\right\}\\
        &= \left(Z^{\ast T}\Zast +\gamastn S(\tilde\beta_n^{\ast(k)})\right)^{-1}Z^{\ast T}L_n(\tilde\alpha_n^{(k+1)}),
    \end{split}
\end{equation}
where $\gamastn \coloneqq \diag(\gamma_{n1},...,\gamma_{nq},\gamma_{n,q+1})$ is the diagonal tuning matrix for the iterated AR estimate of $\beta^\ast$. 

\begin{rem}
The initial ridge estimators $\tilde\alpha_n^{(0)}$ and $\tilde\beta_n^{(0)}$ and their subsequent iterations are non-zero almost surely, thus ensuring $T(\alpha)$ and $S(\bast)$ to be well-defined. 
\end{rem}

\begin{rem}
    The initial tuning matrices $\Psi_n$ and $\Omega^\ast_n$ are defined separately from the tuning matrices for iterated estimates ($\Lambda_n$ and $\gamastn$) because they typically have different orders. We will discuss them further in Section \ref{sec3}.
\end{rem}

As a result of using $l_2$ regularization, the AR admits closed forms, as in the expressions \eqref{eq2.1}, \eqref{eq2.3}, \eqref{eq2.4}, and \eqref{eq2.5}, which makes AR efficiently computable and scalable with $p$ and $q$. However, the algorithm still involves the division of small, non-zero values in the matrices $T(\alpha)$ and $S(\beta^\ast)$ when the estimates of $\alpha$ and $\beta^\ast$ are small, in turn causing arithmetic overflow (Dai et al.(2018) \cite{BAR2018}). One solution given in Frommlet and Nuel (2016) \cite{PLOSONE} to circumvent this issue is by introducing a small perturbation to $T(\alpha)$ and $S(\beta^\ast)$, so that they are instead computed as 
\begin{align}
    T(\alpha) &=\diag\left( (\alpha_1^{2}+\delta_\alpha)^{-1},..,(\alpha_p^{2}+\delta_\alpha)^{-1}\right),
    \nn\\
    S(\bast) &= \diag\left( (\beta_1^{2}+\delta_\beta)^{-1},..,(\beta_q^{2}+\delta_\beta)^{-1},(c^{2}+\delta_\beta)^{-1}\right)
\end{align}
for some small constants $\delta_\alpha, \delta_\beta >0$, which improves numerical stability at the cost of introducing some bias and fitting two extra hyperparameters. 

Later in \Cref{sec4}, we choose to adopt another method proposed by Liu and Li (2016) \cite{LiuLi2016} and Dai et al.(2018) \cite{BAR2018}. By denoting
\[\tilde{X}^{(k)}_n:=D_n^{-1}(\tilde\beta^{(k)}_n)XT^{-\frac{1}{2}}(\tilde\alpha^{(k)}_n) \hspace{0.5cm} \text{and} \hspace{0.5cm} \tilde{Y}^{(k)}_n:=D_n^{-1}(\tilde\beta^{(k)}_n)Y,\] we can write for iterated estimators,
\begin{align*}
    \tilde\alpha^{(k+1)}_n &= \left(X^T D^{-2}_n(\tilde\beta_n^{(k)} )X+\Lambda_nT(\tilde\alpha_n^{(k)})\right)^{-1}X^TD^{-2}_n(\tilde\beta_n^{(k)})Y\\    
    &=\left(T^{\frac{1}{2}}(\tilde\alpha_n^{(k)})\left(T^{-\frac{1}{2}}(\tilde\alpha_n^{(k)})X^TD^{-2}_n(\tilde\beta_n^{(k)})XT^{-\frac{1}{2}}(\tilde\alpha_n^{(k)})+\Lambda_n\right)T^{\frac{1}{2}}(\tilde\alpha_n^{(k)})\right)^{-1}
    \nn\\
    &{}\qquad \times X^TD^{-2}_n(\tilde\beta_n^{(k)})Y\\
    &= T^{-\frac{1}{2}}(\tilde\alpha^{(k)}_n)\left(T^{-\frac{1}{2}}(\tilde\alpha_n^{(k)})X^TD^{-2}_n(\tilde\beta_n^{(k)})XT^{-\frac{1}{2}}(\tilde\alpha_n^{(k)})+\Lambda_n\right)^{-1}
    \nn\\
    &{}\qquad \times T^{-\frac{1}{2}}(\tilde\alpha^{(k)}_n)X^TD^{-2}_n(\tilde\beta_n^{(k)})Y\\
    &= T^{-\frac{1}{2}}(\tilde\alpha^{(k)}_n)\left(\tilde{X}^{(k) T}_n\tilde{X}^{(k)}_n+\Lambda_n\right)^{-1}\tilde{X}^{(k)T}_n\tilde{Y}^{(k)}_n.
\end{align*}
Similarly, by writing \[\tilde{Z}^{\ast(k)}_n:=\Zast S^{-\frac{1}{2}}(\tilde\beta^{\ast(k)}_n),\] 
we get
\[\tilde\beta^{\ast(k+1)}_n=S^{-\frac{1}{2}}(\tilde\beta^{\ast(k)}_n)\left(\tilde{Z}^{\ast(k) T}_n\tilde{Z}^{\ast(k)}_n+\gamastn \right)^{-1}\tilde{Z}^{\ast(k) T}_nL_n(\tilde\alpha^{(k)}_n),\]
allowing us to deal with most stability issues without fitting two extra tuning parameters.

\section{Main results}\label{sec3}
\subsection{Assumptions}\label{sec31}
\begin{assump}\label{(A1)}
There exist positive definite matrices $\Sigma^X$, $\Sigma^X_+$, $\Sigma^X_-$, $\Sigma^Z$ and $\Sigma^{\Zast }$ such that as $n \to \infty$,
\begin{equation}\label{eq3.1}
\begin{gathered}
      \frac{X^TX}{n} \to \Sigma^X, \quad \frac{X^TD_n^2(\beta_0)X}{n} \to \Sigma^X_+, 
      \quad \frac{X^TD_n^{-2}(\beta_0)X}{n} \to \Sigma^X_-, \\
      \frac{Z^TZ}{n} \to \Sigma^Z, \quad \text{ and }\quad \frac{Z^{\ast T}\Zast }{n} \to \Sigma^{\Zast}.
\end{gathered}
\end{equation}
Additionally, there exist constants $K_X > 1$, $K_Z > 1$, and $N\in\mathbb{N}$ for which
\begin{align}
    & \frac{1}{K_X} \leq \inf_{n\ge N} \inf_\beta \lambda_{\min}\left(\frac{X^T D^{-2}_n(\beta)X}{n}\right) \leq  \sup_{n\ge N} \sup_\beta \lambda_{\max}\left(\frac{X^TD^{-2}_n(\beta)X}{n}\right) \leq K_X,
    \nonumber\\
    & \frac{1}{K_Z} \leq \inf_{n\ge N} \lambda_{\min}\left(\frac{Z^{\ast T}\Zast }{n}\right) \leq \sup_{n\ge N} \lambda_{\max}\left(\frac{Z^{\ast T}\Zast }{n}\right) \leq K_Z.
\end{align}
\end{assump}

\begin{assump}\label{(A2)}
    \[\sup_{n} \Vert X_n \Vert < \infty \text{ and }\sup_{n} \Vert \Zast_n \Vert < \infty.\]
\end{assump}

\begin{assump}\label{(A3)}
    The random errors $\epsilon_1, \epsilon_2, ...$ are i.i.d. with $\mathbb{E}\left[\epsilon_1\right] = 0$, $\var[\epsilon_1] = 1$, and $\mathbb{E}\left[|\epsilon_1|^{-s}\right] < \infty$ for all $s \in (-\infty,1)$.
\end{assump}

\begin{assump}\label{(A4)}
    For each $k \geq 0$, there exist some $a \in (0,1)$ and $N$ such that \[\sup_{n>N} \sup_{i} \mathbb{E}\left[\left|e^{\frac{1}{2}Z_i^T\beta_0}\epsilon_i - X_i^T (\tilde\alpha_n^{(k)} - \alpha_0)\right|^{-a}\right] < \infty.\]
\end{assump}

Assumptions \ref{(A1)}, \ref{(A2)} and the mean and variance conditions of Assumption \ref{(A3)} are standard to make. The negative moments condition in Assumption \ref{(A3)} requires the noise distribution not to have a positive mass at $0$, as that prohibits the consideration of logarithms. In addition, this assumption guarantees the existence of logarithmic moments $\mathbb{E}[\log(\epsilon)]$ and $\mathbb{E}[\log(\epsilon^2)]$. Assumption \ref{(A4)} is technical and is admittedly challenging to verify, but is needed for proving asymptotic results of $\beta$, We will later present a heuristic argument as to why we believe this assumption is still reasonable to make. To do so, we refer to Khuri and Casella (2002) \cite{inversemoments}, which establishes a link between negative moments of random variables and their densities (if they exist). Their results and further explanation will be given after stating Theorem \ref{thm1}. 

\subsection{Asymptotic results of $\tilde\alpha_n^{(0)}$ and $\tilde\beta_n^{\ast(0)}$}
We provide an asymptotic result for our initial estimators. As we construct an iterative scheme, the following convergences will play a vital role in asymptotic results for subsequent estimators. The following result for $\alpha$ is well known and is present in literature such as Knight and Fu (2000) \cite{KnightFu2000}. The asymptotic results for $\beta$ are first proved here to the authors' understanding.
\begin{thm}\label{thm1}
Under Assumptions \ref{(A1)} to \ref{(A3)}, if $\Psi_n/\sqrt{n} \to \Psi_0$ as $n \to \infty$ for some constant matrix $\Psi_0$, then $\tilde\alpha_n^{(0)}$ is asymptotically normal:
\begin{equation}\label{eq3.3}
    \sqrt{n}\left(\tilde\alpha^{(0)}_n-\alpha_0 \right) \cil N \left(-(\Sigma^X)^{-1}\Psi_0\alpha_0,\, (\Sigma^X)^{-1}\Sigma^X_+(\Sigma^X)^{-1}\right).
\end{equation} 
If Assumption \ref{(A4)} also holds and $\Omega_n^\ast/\sqrt{n} \to 0$, then $\bastinit$ is a consistent estimator of $\basttrue$, and for all $c \in [0,1/2)$, 
\begin{equation}\label{eq3.4}
    n^c\left(\bastinit - \basttrue\right) = \mathcal{O}_p(1).
\end{equation}
In addition, if $\mathbb{E}\left[|\epsilon_1|^{-(1+\tau)}\right] < \infty$ for some $\tau > 0$, then $\bastinit$ is asymptotically normal:
\begin{equation}\label{eq3.5}
    \sqrt{n}\left(\bastinit - \basttrue\right) \cil N\left(0, \left(\Sigma^{\Zast }\right)^{-1}\var\left[\log(\epsilon_1^2)\right]\right).
\end{equation}
\end{thm}
Although we have stated a general result for $\Psi_n/\sqrt{n} \to \Psi_0$, we will focus on $\Psi_n/\sqrt{n} \to 0$ for the remainder of this paper because we require $\tilde\alpha_n^{(0)}$ to be asymptotically unbiased. 

\begin{rem}
    While \eqref{eq3.5} holds only under the condition $\mathbb{E}\left[|\epsilon_1|^{-(1+\tau)}\right] < \infty$, our numerical experiments in Section \ref{hm:ss_moment} suggest that this may be redundant.
\end{rem}
To better understand the negative moments condition needed for \eqref{eq3.5}, as well as Assumptions \ref{(A3)} and \ref{(A4)}, we note the following facts (see Khuri and Casella (2002) \cite{inversemoments} for related detail): 
if $|\epsilon_1|$ possesses a density function $f$, then 
\begin{enumerate}[label=(KC\arabic*)]
    \item If $f$ is bounded near $0$ then $\mathbb{E}\left[|\epsilon_1|^{-s}\right] < \infty$ for all $0 < s < 1$.
    \item If $f(0) > 0$ then $\mathbb{E}\left[|\epsilon_1|^{-1}\right] = \infty$.
    \item If $\ds{\lim_{x \to 0^+} x^{-\delta}f(x) = K}$ for some $\delta >0$ and $K \in [0,\infty)$, then $\mathbb{E}\left[|\epsilon_1|^{-(1+\tau)}\right]<\infty$ for all $0 < \tau < \delta$.
\end{enumerate}

Referring to these conditions, it is straightforward to see that most errors, including normal errors, satisfy the condition needed for $\mathbb{E}\left[|\epsilon_1|^{-s}\right] < \infty$ for all $0 < s < 1$ in Assumption \ref{(A3)}, but fail to possess an inverse first moment, which is why \eqref{eq3.4} and \eqref{eq3.5} are presented separately. Still, there are families of distribution that satisfy the condition in (KC3), for instance, one can show that the log-normal distribution (reflected about $x=0$ and scaled) is a class of errors that satisfy the stricter conditions. 

Returning to Assumption \ref{(A4)}, the existence of the inverse moments can be shown using (KC1), by showing that for each $k$, there is some sufficiently large $N$ so that
\[ \limsup_{u\to 0}\sup_{n>N} \sup_{i = 1,...,n} f^{(k)}_{n,i}(u) < \infty, \]
where $f^{(k)}_{n,i}(u) $ denotes the density of $\left|e^{\frac{1}{2}Z_i^T\beta_0}\epsilon_i - X_i^T (\tilde\alpha_n^{(k)} - \alpha_0)\right|$ if we assume that they exist. 

\begin{rem}
If we assume the correctly specified parametric model for the noise distribution of $\ep_1$, then the rate of convergence \eqref{eq3.4} would be enough to proceed with constructing the well-known Newton-Raphson type one-step estimator to improve asymptotic efficiency (see Zacks (1971) \cite[Section 5.5]{Zac71}). We do not go into further details here.
\end{rem}

\subsection{Notation Under Sparsity}\label{sec3.3}
For clarity, we introduce additional notation for subsequent parts targeted at sparse models, which is a major advantage of the AR over other $l_2$ regularization schemes. Suppose exactly $0 \leq p_0 \leq p$ and $0 \leq q_0 \leq q+1$ components from $\alpha_0$ and $\bast_0$ are non-zero. We will denote using the ($\star$) and ($\circ$) subscripts for the non-zero and zero components respectively. Upon permuting the components, $\alpha_0 = (\alpha_{0\star},\alpha_{0\circ}) = (\alpha_{0\star},0,..,0)$ and $\basttrue = (\bast_{0\star},\bast_{0\circ}) = (\bast_{0\star},0,..,0)$. Thus we will write $\alpha^{(k)}_n = (\alpha^{(k)}_{n\star},\alpha^{(k)}_{n\circ})$ and $\beta^{\ast(k)}_n = (\beta^{\ast(k)}_{n\star},\beta^{\ast(k)}_{n\circ})$ in the same manner, despite $p_0$ and $q_0$ being unknown. Moreover, the symmetric $p$ by $p$ matrices such as 
\begin{center}
    $\ds{\tilde\Sigma_{n+}^{(0)} :=\frac{X^TD^2_n(\tilde\beta^{(0)}_n)X}{n}}$ \quad and \quad $\ds{\tilde\Sigma_{n-}^{(0)} :=\frac{X^TD^{-2}_n(\tilde\beta^{(0)}_n)X}{n}}$
\end{center} 
will be written in the form of
\begin{equation}\label{eq3.6}
    \Sigma = 
\begin{pmatrix}
(\Sigma)_{\star\star} & (\Sigma)_{\star\circ} \\
(\Sigma)_{\star\circ}^T & (\Sigma)_{\circ\circ}
\end{pmatrix},
\end{equation}
where the top left block will have dimension $p_0 \times p_0$, corresponding to the nonzero entries of $\alpha$; this is a generic notation for block decomposition. Similar subscripts will be adopted for $(q+1)$ by $(q+1)$ matrices such as $Z^{\ast T}\Zast/n$, with the top left block being size $q_0 \times q_0$. Diagonal matrices such as $\Lambda_n$ and $\gamastn$ will also be written as 
\begin{equation}\label{eq3.7}
    \Lambda_n =
\begin{pmatrix}
\Lambda_{n\star} & 0 \\
0 & \Lambda_{n\circ}
\end{pmatrix}
 \text{ and } \gamastn =
\begin{pmatrix}
\Gamma^\ast_{n\star} & 0 \\
0 & \Gamma^\ast_{n\circ}
\end{pmatrix}
\end{equation}
with the top left block having dimensions $p_0$ by $p_0$ and $q_0$ by $q_0$ respectively. We further introduce the functions
\begin{align}
    & T_\star(\alpha) = \diag(\alpha^{-2}_1,...,\alpha^{-2}_{p_0}), \quad 
    T_\circ(\alpha) = \diag(\alpha^{-2}_{p_0+1},...,\alpha^{-2}_p), \nn\\
    & S_\star(\bast) =\diag(\beta^{\ast-2}_1,...,\beta^{\ast-2}_{q_0}), \quad 
    S_\circ(\bast) = \diag(\beta^{\ast-2}_{q_0+1},...,\beta^{\ast-2}_{q+1})\nn
\end{align}
corresponding to the zero and non-zero entries of $\alpha_0$ and $\bast_0$ respectively. 

\subsection{Asymptotic Results of $\tilde\alpha_n^{(k)}$ and $\tilde\beta_n^{\ast (k)}$}

The next theorem provides insight into the asymptotic behaviour of estimators after a fixed number of iterations $k$. In particular, we will see how there is reduced asymptotic covariance compared to the initial estimators of Theorem \ref{thm1}. In the remainder of the paper we will denote $\tilde u_n^{(k)} \coloneqq \sqrt{n}\left(\tilde\alpha_n^{(k)} - \alpha_0\right)$ and $\tilde v_n^{(k)} \coloneqq \sqrt{n}\left(\tilde\beta_n^{\ast(k)} - \basttrue \right)$.
\begin{thm}\label{thm2}
Assume $\Lambda_n = \begin{pmatrix}
\Lambda_{n\star} & 0 \\
0 & \Lambda_{n\circ}
\end{pmatrix} \to \begin{pmatrix}
\Lambda_{0\star} & 0 \\
0 & \Lambda_{0\circ}
\end{pmatrix} = \Lambda_0$ for some constant matrix $\Lambda_0$. Under Assumptions \ref{(A1)} to \ref{(A3)}, for each integer $k \geq 1$,
\begin{equation}\label{eq3.8}
\tilde u_n^{(k)} = \left(\Sigma_-^X  + \begin{pmatrix}
    0 & 0 \\
    0 & \Lambda_{0\circ}T_\circ(\tilde u_{n}^{(k-1)})
    \end{pmatrix}\right)^{-1}\Pi^X_n + \smallO_p(1),
\end{equation}
where $\Pi^X_n \cil N\left(0,\Sigma_-^X \right)$.
    
If Assumption \ref{(A4)} holds and $\gamastn/\sqrt{n} \to 0$, then for all $k \geq 1$ and $c \in [0,1/2)$, 
\begin{equation}\label{eq3.9}
    n^c\left(\tilde\beta^{\ast(k)}_n - \basttrue\right) = \mathcal{O}_p(1).
\end{equation}
Additionally, if $\mathbb{E}\left[|\epsilon_1|^{-(1+\tau)}\right] < \infty$ for some $\tau > 0$, and 
\begin{equation}\nn
    \gamastn = \begin{pmatrix}
\Gamma^\ast_{n\star} & 0 \\
0 & \Gamma^\ast_{n\circ}
\end{pmatrix} \to \begin{pmatrix}
\Gamma^\ast_{0\star} & 0 \\
0 & \Gamma^\ast_{0\circ}
\end{pmatrix} =  \Gamma_0^\ast
\end{equation}
for some constant matrix $\Gamma_0^\ast$, then
\begin{equation}\label{eq3.10}
    \tilde v_n^{(k)} = \left(\Sigma^{\Zast} + \begin{pmatrix}
    0 & 0 \\
    0 & \Gamma^\ast_{0\circ}T_\circ(\tilde v_{n}^{(k-1)})
    \end{pmatrix}\right)^{-1}\Pi^Z_n + \smallO_p(1),
\end{equation}
where $\Pi^Z_n \cil N\left(0,\Sigma^{\Zast} \var\left[\log(\epsilon_1^2) \right]\right)$.
\end{thm}

\begin{rem}\nn
If $\alpha_0$ does not contain any zero components, then \eqref{eq3.8} simplifies to 
\begin{equation}\label{eq3.11}
\sqrt{n}\left(\tilde\alpha^{(k)}_n-\alpha_0 \right) \cil N \left(0,(\Sigma^X_-)^{-1}\right).
\end{equation}
Similarly \eqref{eq3.10} simplifies to \eqref{eq3.5} when $\beta_0^\ast$ contains no zero components. 
\end{rem}
\begin{rem}
    In \eqref{eq3.8} and \eqref{eq3.10}, we have also chosen the tuning matrices $\Lambda_n$ and $\Gamma^\ast_n$ to converge at a slower rate than those of $\Psi_n$ and $\Omega_n$ in \Cref{thm1}. This is mainly attributed to the influence of the additional tuning terms $T(\alpha)$ and $S(\beta^\ast)$. We may refer to Dai et al.(2018) \cite{BAR2018} for related sparse asymptotic results like the oracle property in the case $k \to \infty$.
\end{rem}

In Theorem \ref{thm2}, we have not identified the weak limits of $\tilde u_n^{(k)}$ and $\tilde v_n^{(k)}$. Yet, the matrices 
\begin{equation}\nn
 \begin{pmatrix}
    0 & 0 \\
    0 & \Lambda_{0\circ}T_\circ(\tilde u_{n,\circ}^{(k-1)})
    \end{pmatrix}
    \qquad\text{and}\qquad
    \begin{pmatrix}
    0 & 0 \\
    0 & \Gamma^\ast_{0\circ}T_\circ(\tilde v_{n,\circ}^{(k-1)})
    \end{pmatrix}   
\end{equation}
    are non-negative definite, which implies
        \[\left\Vert \left(\Sigma^X_- + \begin{pmatrix}
        0 & 0 \\
        0 & \Lambda_{0\circ}T_\circ(\tilde u_{n,\circ}^{(k-1)})
        \end{pmatrix}\right)^{-1} \right\Vert \leq \left \Vert \left(\Sigma^X_- \right)^{-1} \right \Vert\]  and  \[\left\Vert \left(\Sigma^{\Zast} + \begin{pmatrix}
        0 & 0 \\
        0 & \Gamma^\ast_{0\circ}T_\circ(\tilde v_{n,\circ}^{(k-1)})
        \end{pmatrix}\right)^{-1} \right\Vert \leq \left \Vert \left(\Sigma^{\Zast} \right)^{-1} \right \Vert.\]
Therefore, we can infer that the AR estimators are $\sqrt{n}$-tight and gain reduced covariance by penalization. However, Theorem \ref{thm2} alone is insufficient in describing all the benefits of AR as it does not justify the motivation to perform extra iterations. Since the changes in distributions from extra iterations depend solely on $\begin{pmatrix}
    0 & 0 \\
    0 & \Lambda_{0\circ}T_\circ(v_{n,\circ}^{(k-1)})
    \end{pmatrix}$ and $\begin{pmatrix}
    0 & 0 \\
    0 & \Gamma^\ast_{0\circ}T_\circ(u_{n,\circ}^{(k-1)})
    \end{pmatrix}$, understanding the behaviour of sparse estimates is crucial, which motivated us to give the next result.

\begin{thm}\label{thm3}
Under Assumption \ref{(A1)}, if $\min_{1 \leq j \leq p} \lambda_{nj} \to \infty$ and $\Lambda_n/\sqrt{n} \to 0$, then for any $k \geq 0$, 
\begin{equation}\label{eq3.12}
\dfrac{\Vert \tilde\alpha^{(k+1)}_{n\circ} \Vert}{\Vert \tilde\alpha^{(k)}_{n\circ}  \Vert} \cip 0 \quad \text{as } n \rightarrow \infty.
\end{equation}
If $\min_{1 \leq j \leq q+1} \gamma_{nj}/n^c \to \infty$ for some $c >0$ and $\gamastn/\sqrt{n} \to 0$, then for any $k \geq 0$,
\begin{equation}\label{eq3.13}
    \dfrac{\Vert \tilde\beta^{\ast(k+1)}_{n\circ} \Vert}{\Vert \tilde\beta^{\ast(k)}_{n\circ}  \Vert} \cip 0 \quad \text{as } n \rightarrow \infty.
\end{equation}
If $\mathbb{E}\left[|\epsilon_1|^{-(1+\tau)}\right] < \infty$ for some $\tau > 0$, then we can replace $\min_{1 \leq j \leq q+1} \gamma_{nj}/n^c \to \infty$ above with $\min_{1 \leq j \leq q+1} \gamma_{nj} \to \infty$.
\end{thm}

We note that the convergence rates of $\Lambda_n$ and $\Gamma_n^\ast$ in Theorems \ref{thm2} and \ref{thm3} are different, as Theorem \ref{thm2} assumes the convergence of hyperparameters to a constant, while Theorem \ref{thm3} assumes the divergence of them. However, since $\Lambda_n$ and $\gamastn$ can diverge arbitrarily slowly, both results can often be observed simultaneously in practice. 

\section{Simulation}\label{sec4}
\subsection{Parameter Estimation and the effect of iterations}
This section aims to illustrate the theoretical results presented in Section \ref{sec3} and provide numerical evidence to generalize them further. We first illustrate the shrinkage and distributional properties of the AR described by the theorems. Assume the model \eqref{eq1.1} and let $p = q = 20$. We generate data $X_i$ and $Z_i$ ($i = 1,...,n$) as $n$ independent copies of $X_0 \sim N(0, \Sigma)$ and $Z_0 \sim N(0,\Sigma)$, where $X_0$ and $Z_0$ are independent and $\Sigma$ is the covariance matrix with 1 in the diagonals and 0.4 in the off-diagonals, simulating correlated features. We consider three types of noise distributions - standard normal - $N(0,1)$, Laplace with location and scale parameters 0 and 1 - Laplace$(0,1)$, and $t$ with degrees of freedom 3 - $t(3)$. For the standard normal, the expectation $\mathbb{E}[\log(\epsilon^2)]$ can be calculated to be (Harvey (1976) \cite{Harvey1976}) 
\begin{align*}
    \mathbb{E}\left[\log(\epsilon_1^2)\right] = \text{Digamma}\left(\frac{1}{2}\right) - \log\left(\frac{1}{2}\right) \approx -1.2704.
\end{align*}
For the Laplace and $t$ distributions, we numerically calculate them to be approximately -1.154 and -0.9014 respectively. The true signals are set to be $\alpha_0 = \beta_0 = (0.1,0.2,...,1,0,...,0)$ and $\beta_0^\ast = (\beta_0,\mathbb{E}[\log(\epsilon^2)])$. We compare results when $n = 100, 1000$ for the AR with $k = 0, 2, 5, 10$ and the broken adaptive ridge (BAR) over 1000 independent trials, showing how the AR evolves with iterations. The BAR estimator is obtained as a numeric fixed point of the iteration process when the norm of the difference of consecutive estimates falls under a certain threshold.

We will set all tuning matrices to be a multiple of the identity matrix so that $\Psi_n = \psi I_p$ and define $\lambda$, $\gamma$, and $\omega$ similarly. We ran extensive grid searches for the hyperparameters using cross-validation (CV) and observed that the CV error is stable for a wide range of $\psi$ and $\omega$. The error is also close to minimal for $\lambda,\gamma \in (10^{-3},1)$. To simplify computations, instead of searching for the optimal hyperparameters at each iteration, we will set $\psi = \omega = \sqrt{n}$, which is the rate discussed in Theorem \ref{thm1}, and set $\lambda = \gamma = 0.1\log(n)$, matching the rate in Theorem \ref{thm3}.

Tables \ref{tab1} and \ref{tab2} show the effect iterations have in estimating both sparse and non-sparse parameters for $n = 100, 1000$ and three types of error distribution. For estimates of $\alpha$, the mean squared error (MSE) drastically drops from $k = 0$ to $k = 2$ regardless of the error distribution. This can be attributed to factoring in heterogeneous error terms. For non-sparse components of both $\alpha$ and $\beta$, the error can drop further in some scenarios but in other cases, such as when $n = 100$ and $\epsilon \sim t(3)$, the minimal empirical error is observed when $k = 5$ for $\alpha$. It is also unsurprising to see the improvement in accuracy by increasing $n = 100$ to $n = 1000$, as we numerically verify results from Theorems \ref{thm1} and \ref{thm2}. 

The effect of iterations becomes more apparent when inspecting errors for sparse estimates, where the error generally decreases as we iterate more, with a couple of exceptions. This is due to sparse estimates being increasingly shrunk towards 0, a phenomenon best explained by Theorem \ref{thm3}. To further visualize this, we can refer to tables \ref{tab3} and \ref{tab4}, where we observe the decrease of the medians of the absolute values of all sparse estimates for both $\alpha$ and $\beta$, for all types of error distributions. Crucially, while the AR estimates are never sparse, by implementing an arbitrary sparsity threshold, we can control the degree of sparsity by controlling the number of iterates $k$, allowing for a higher degree of flexibility for the user. As for the BAR, as we regarded values under $10^{-100}$ to be 0, almost all values in both tables are recorded to be 0.

We can also refer to Figure \ref{fig1} to better visualize the behaviour of sparse estimates by iterating. Namely, observe that as the number of iterations increases, sparse estimates of $\alpha_{0,11} = 0$ become more concentrated near 0, showing the same phenomena as Table \ref{tab3}.
\begin{rem}
Since we are considering large finite sample behavior, it is possible that a few estimates for zero components do not converge and instead are much further away from zero than the rest. In such cases, if we use the mean to evaluate the extent of shrinkage, such outliers would prevent any meaningful comparison. Instead, using the median easily circumvents this issue.
\end{rem}

From a computational standpoint, it may be more intuitive to compute the BAR instead of tuning for an additional hyperparameter in the number of iterations $k$. Yet, if we carefully inspect the MSEs in tables $1$ and $2$, we see that minimal error is often achieved for $k = 10$ instead of the BAR, especially for the estimates of $\beta^\ast$. We suspect part of the reason is due to the occasional divergence of BAR estimates in finite samples, while it is also possible that the BAR estimates are too aggressive in pursuing a sparse solution. In that sense, depending on the criterion set for convergence, keeping the number of iterations $k$ finite can greatly reduce the computational burden without sacrificing performance. 

\begin{table}[h]
    \centering
        \caption{The mean squared error (MSE) of the AR estimates of $\alpha$ for four different numbers of iterations ($k = 0,2,5,10$) and the broken adaptive ridge (BAR), averaged over 1000 independent trials. We consider two sample sizes ($n = 100, 1000$) and $N(0,1)$ (N), Laplace$(0,1)$ (L) and $t(3)$ (T) noise distributions. The non-sparse components (NS) of $\alpha$ consist of the first 10 components, where $\alpha_{0\star} = (0.1,0.2,...,1)$, and the average MSEs are shown in the top 6 rows. The bottom 6 rows show results for the remaining 10 sparse components (S).} 
    \begin{tabular}{|c|c|c|c|c|c|c|}
        \hline
         &$n,\mathcal{L}(\epsilon)$ & $k = 0$ & $k = 2$& $k = 5$ & $k = 10$ & BAR \\
         \hline
        \multirow{6}{*}{NS}
        &$100$, N& 7.38 & 0.0846 & 0.0188 & 0.0172 & 0.0182\\
        &$1000$, N &0.968 & 5.6 $\times 10^{-4}$ & 1.59 $\times 10^{-4}$ & 1.59 $\times 10^{-4}$ & 1.59 $\times 10^{-4}$\\
        &$100$, L&14.8 & 0.152 & 0.0329 & 0.0337 & 0.0354\\
        &$1000$, L& 1.88 & 8.07$\times 10^{-4}$ & 1.6$\times 10^{-4}$ & 1.42$\times 10^{-4}$ & 1.41$\times 10^{-4}$\\
        &$100$, T& 19.7 & 0.239 & 0.0444 & 0.0392 & 0.0415\\
        &$1000$, T& 2.74 & 0.00385 & 1.51$\times 10^{-4}$ & 1.1$\times 10^{-4}$ & 1.1$\times 10^{-4}$\\
        \hline
        \multirow{6}{*}{S}
        &$100$, N & 6.86 & 0.0601 & 0.00886 & 0.0079 & 0.00815\\
        &$1000$, N&0.844 & 7.49 $\times 10^{-5}$ & 2.51 $\times 10^{-5}$ & 2.55 $\times 10^{-5}$ & 2.52 $\times 10^{-5}$\\
        &$100$, L&13.6 & 0.113 & 0.0205 & 0.0199 & 0.0205\\
        &$1000$, L&1.66 & 1.28$\times 10^{-4}$ & 2.58 $\times 10^{-5}$ & 2.29 $\times 10^{-5}$ & 2.27 $\times 10^{-5}$\\
        &$100$, T&17.8 & 0.212 & 0.0281 & 0.0264 & 0.0281\\
        &$1000$, T& 2.5 & 0.0026 & 4.49$\times 10^{-5}$ & 3.68$\times 10^{-5}$ & 3.65$\times 10^{-5}$\\
        \hline
    \end{tabular}
    \label{tab1}
\end{table}

\begin{table}[h]
    \centering
        \caption{The MSE of BAR and AR ($k = 0,2,5,10$) estimates for $\beta^\ast$ are recorded for $n = 100, 1000$ and three types of error distributions (N,L,T) over 1000 independent trials. The MSEs on the top 6 rows are for the 11 non-sparse components (NS) of $\beta^\ast$, which are $\beta^\ast_{0\star} = (0.1,0.2,...,1,\mathbb{E}[\log(\epsilon^2)])$, where the value of final component depends on the error distribution. Results for the remaining 10 sparse components (S) are shown in the bottom 6 rows. }
    \begin{tabular}{|c|c|c|c|c|c|c|}
        \hline
         &$n,\mathcal{L}(\epsilon)$ & $k = 0$ & $k = 2$& $k = 5$ & $k = 10$ & BAR\\
         \hline
        \multirow{6}{*}{NS}
        &$100$, N&1.58 & 0.25 & 0.282 & 0.435 & 0.476\\
        &$1000$, N &1.01 & 0.013 & 0.01 & 0.00994 & 0.00989\\
        &$100$, L& 1.83 & 0.311 & 0.307 & 0.466 & 0.514\\
        &$1000$, L& 1.26 & 0.0199 & 0.0132 & 0.0129 & 0.0128\\
        &$100$, T&1.74 & 0.295 & 0.305 & 0.475 & 0.535\\
        &$1000$, T&1.26 & 0.022 & 0.012 & 0.0115 & 0.0115\\
        \hline
        \multirow{6}{*}{S}
        &$100$, N &0.0766 & 0.0881 & 0.128 & 0.181 & 0.191\\
        &$1000$, N& 0.00885 & 0.00435 & 0.00522 & 0.0054 & 0.00539\\
        &$100$, L& 0.075 & 0.0995 & 0.168 & 0.239 & 0.256\\
        &$1000$, L& 0.00918 & 0.00597 & 0.0076 & 0.00789 & 0.00787\\
        &$100$, T&0.0775 & 0.1 & 0.156 & 0.225 & 0.243\\
        &$1000$, T&0.00916 & 0.0052 & 0.00635 & 0.00663 & 0.00661\\
        \hline
    \end{tabular}
    \label{tab2}
\end{table}

\begin{figure}
\begin{subfigure}{0.9\textwidth}
    \centering
    \includegraphics[width= 0.31\linewidth]{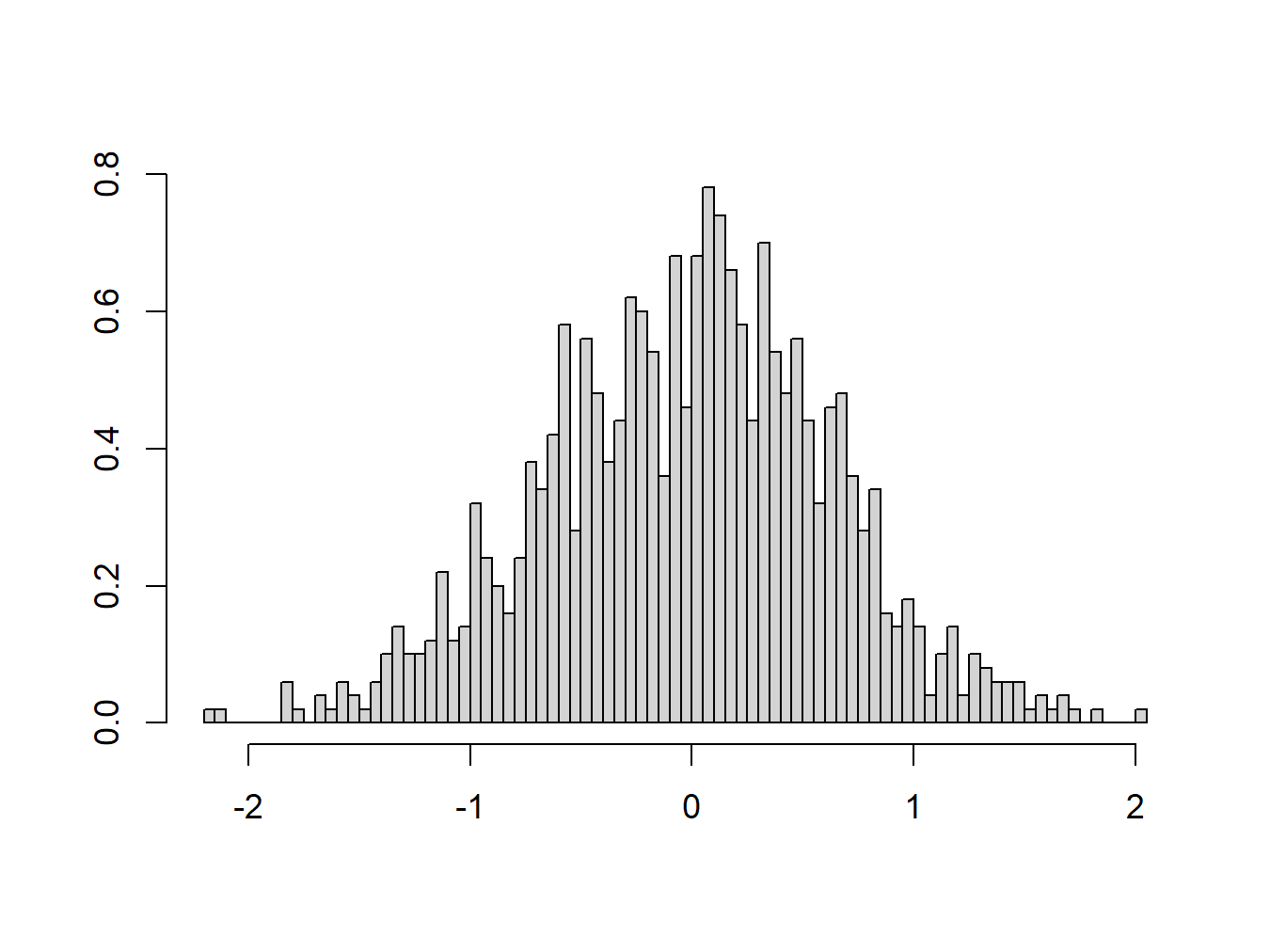}
    \includegraphics[width= 0.31\linewidth]{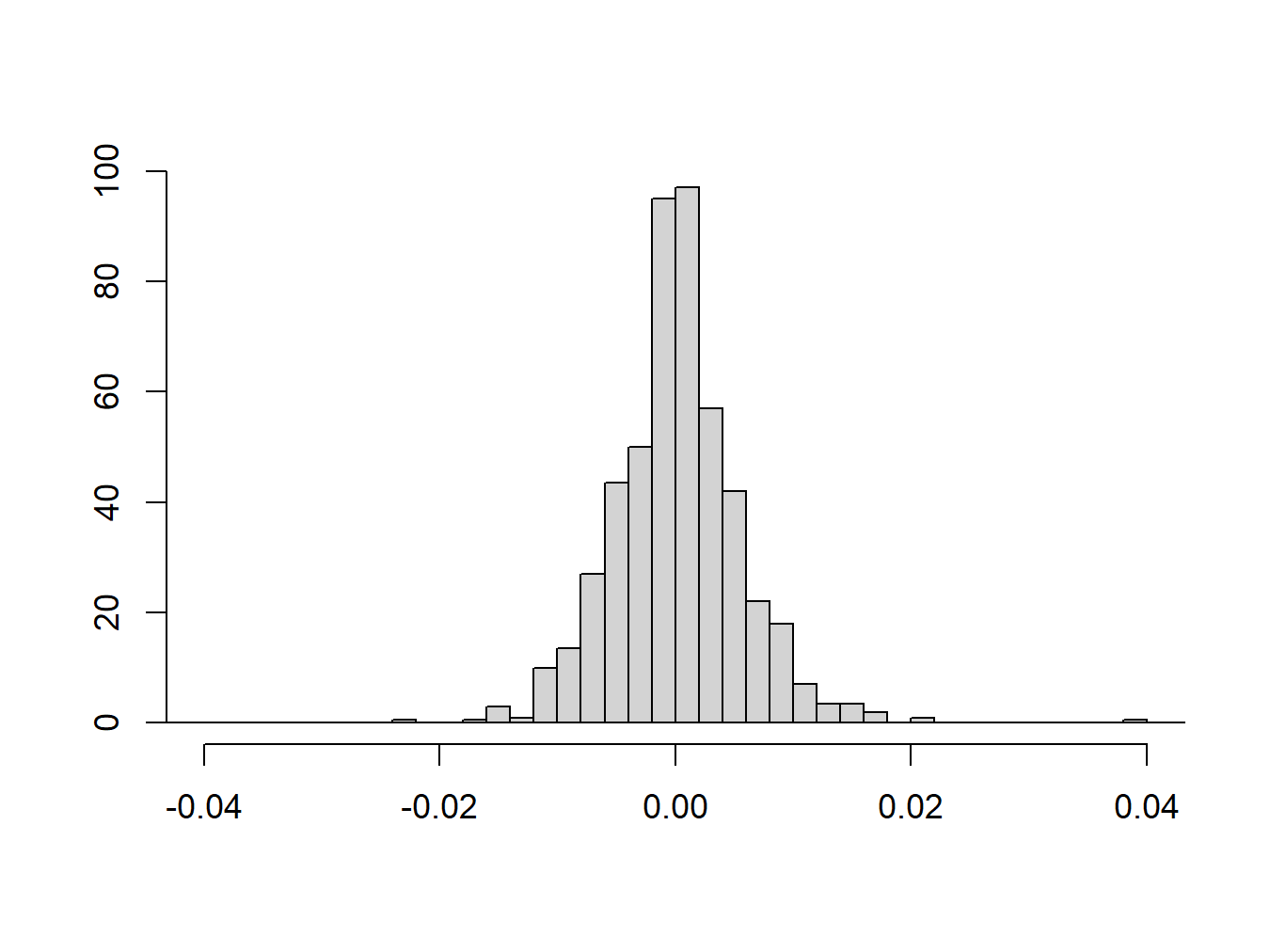}
\end{subfigure}

\begin{subfigure}{0.9\textwidth}
    \centering
    \includegraphics[width= 0.31\linewidth]{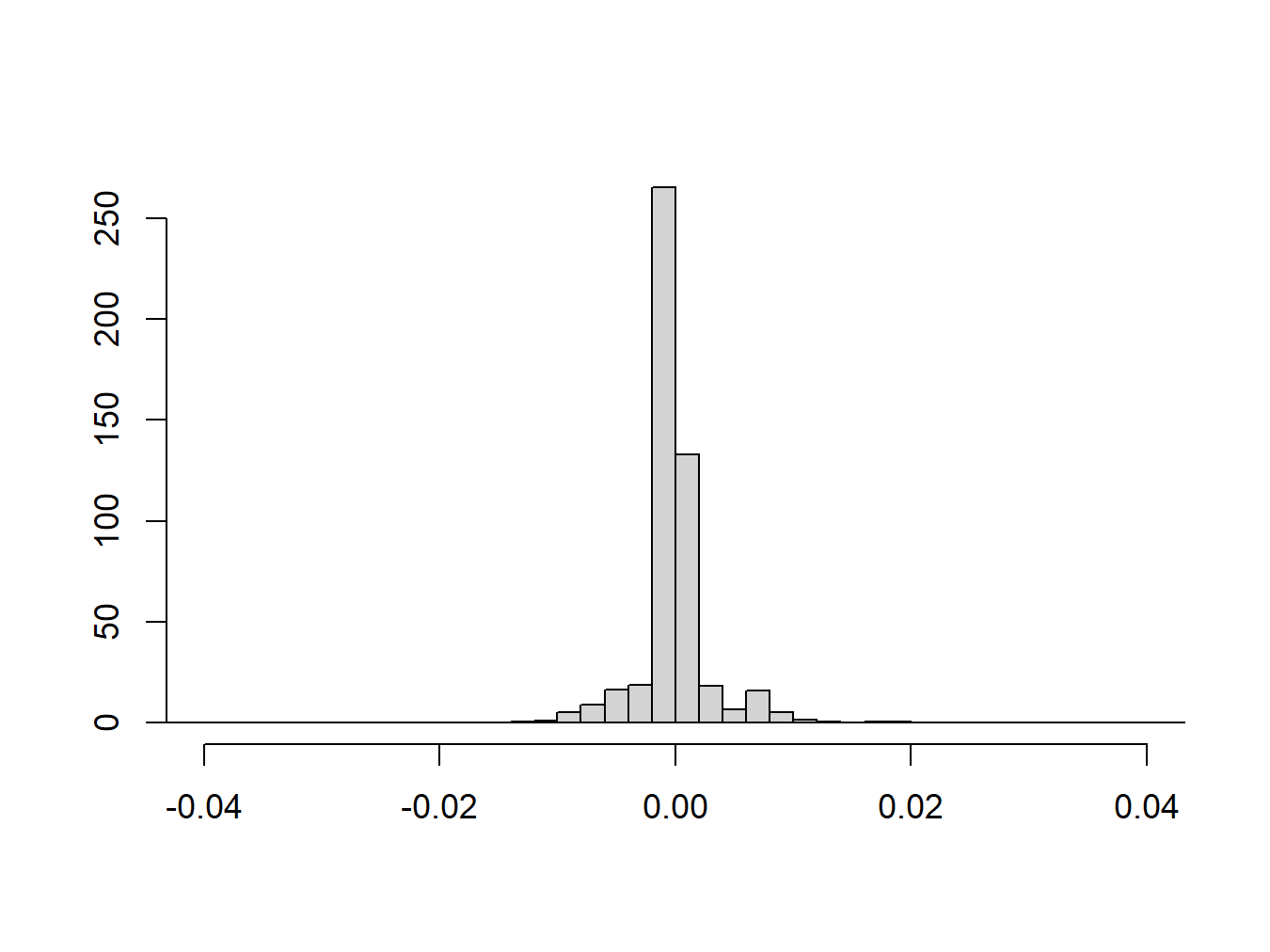}
    \includegraphics[width= 0.31\linewidth]{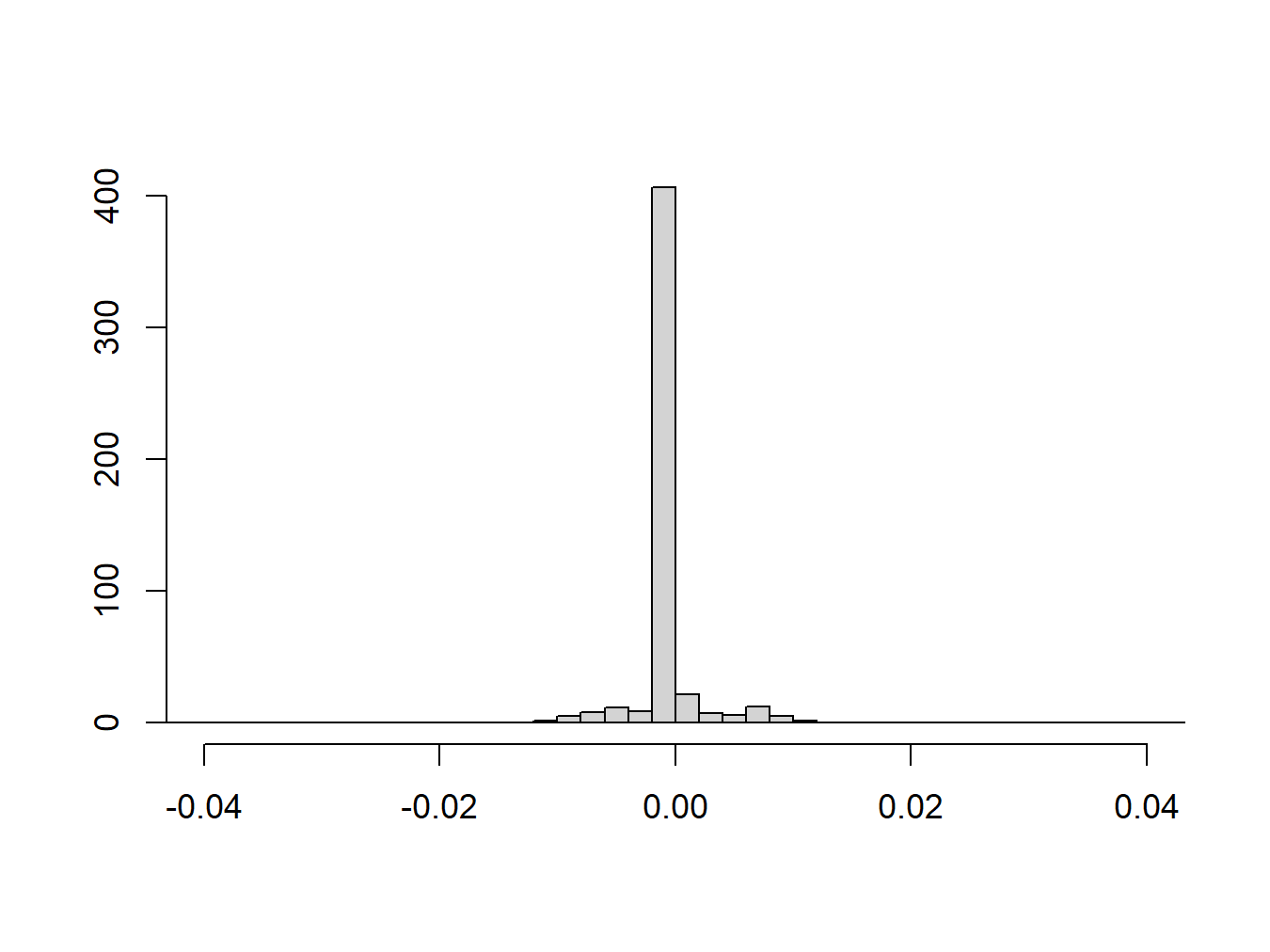}
    \includegraphics[width= 0.31\linewidth]{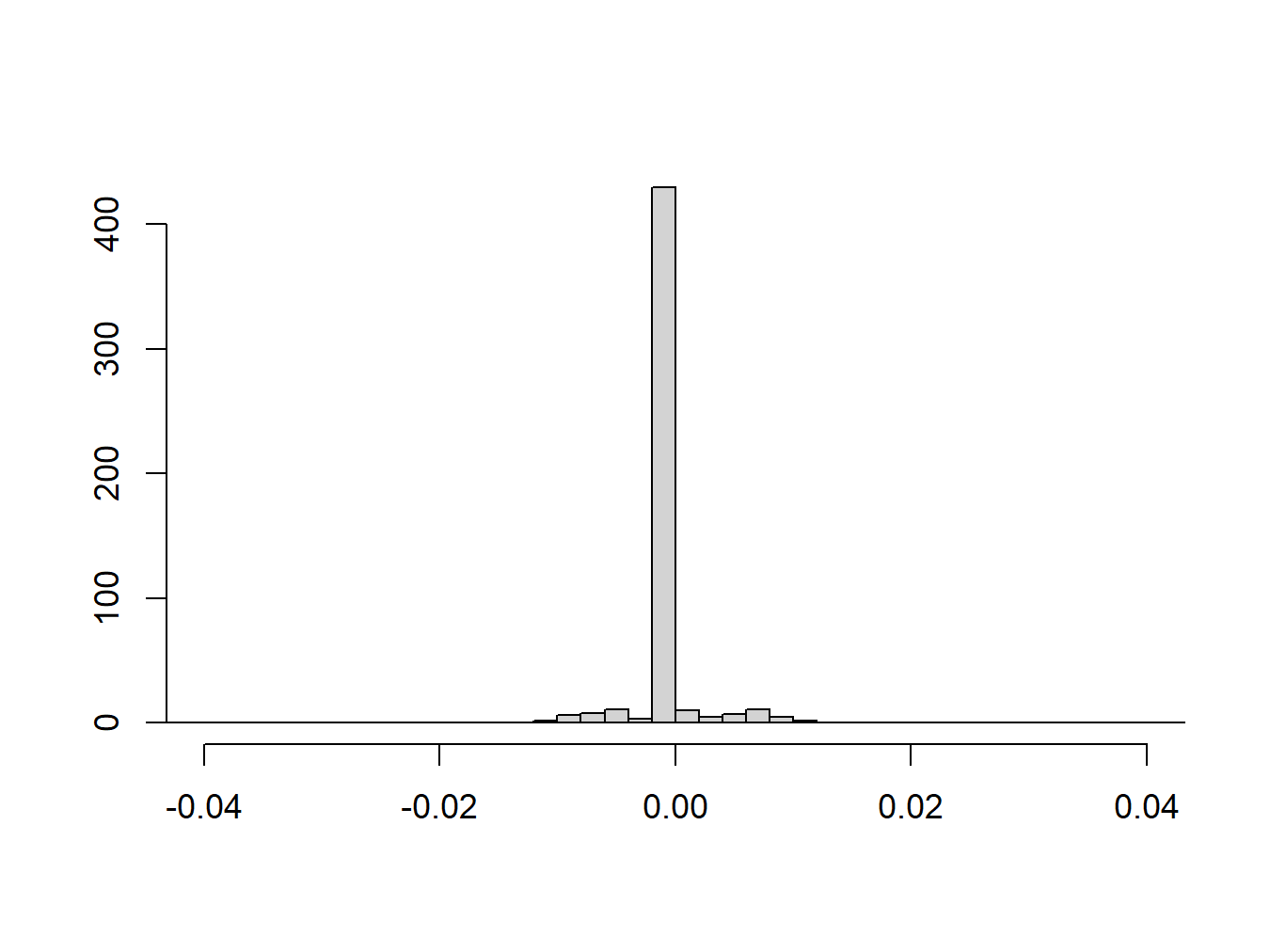}
\end{subfigure}
    \caption{The histograms of 1000 estimates for $\alpha_{0,11} = 0$, showing that as the number of iterates increases, the estimates for sparse parameters become more centred around 0. Histograms are generated with $n = 1000$ under normal errors. The histograms of the AR ($k = 0$, $k = 2$) are on the top left and top right respectively, while the bottom row, from left to right shows the AR ($k = 5$, $k = 10$) and the BAR. The corresponding histograms for $\beta$ are the same so are omitted here.}
    \label{fig1}
\end{figure}

\begin{table}[h]
    \centering
        \caption{This table shows the median of the absolute valued AR/BAR estimates for the 10 sparse signals of $\alpha$ under three noise distributions (N,L,T). This illustrates the shrinkage effect of Theorem \ref{thm3} as we increase the number of iterations, reflected by the decreasing values. The proximity to 0 for AR estimates also depends on the number of samples $n$. Values below $10^{-100}$ are treated as numeric zeros.}
    \begin{tabular}{|c|c|c|c|c|c|c|}
        \hline
         $n,\mathcal{L}(\epsilon)$ & $k = 0$ & $k = 2$& $k = 5$ & $k = 10$ & BAR\\
         \hline
        $100$, N & 1.28 & 0.0969 & 0.0107 & 3.27$\times 10^{-19}$& 0\\
        $1000$, N& 0.592 & 0.00261 & 4.58 $\times 10^{-6}$ & 7.23 $\times 10^{-99}$ & 0\\
        $100$, L&1.62 & 0.146 & 0.0347 & 2.25 $\times 10^{-5}$ & 0\\
        $1000$, L& 0.802 & 0.00435 & 1.34$\times 10^{-4}$ & 5.13 $\times 10^{-48}$ & 0\\
        $100$, T&1.72 & 0.154 & 0.0419 & 0.00132 & 0\\
        $1000$, T&0.885 & 0.00556 & 0.000561 & 5.73$\times 10^{-26}$ & 0\\
        \hline
    \end{tabular}
    \label{tab3}
\end{table}

\begin{table}[h]
    \centering
        \caption{This table records the median of the absolute valued AR/BAR estimates for the 10 sparse signals of $\beta^\ast$ under three noise distributions (N,L,T), as it shows the same trend as table \ref{tab3}. Values below $10^{-100}$ are treated as numeric zeros.}
    \begin{tabular}{|c|c|c|c|c|c|c|}
        \hline
         $n,\mathcal{L}(\epsilon)$ & $k = 0$ & $k = 2$& $k = 5$ & $k = 10$& BAR\\
         \hline
        100, N&0.19 & 0.112 & 0.0346 & 1.42$\times 10^{-13}$& 0\\
        1000, N& 0.0621 & 0.0214 & 4.57$\times 10^{-4}$ & 2.79 $\times 10^{-60}$ & 0\\
        100, L&0.188 & 0.119 & 0.0519 & 4.25$\times 10^{-9}$ & 0\\
        1000, L& 0.0647 & 0.0271 & 0.00378 & 4.7 $\times 10^{-29}$ & 0\\
       100, T&0.192 & 0.121 & 0.0465 & 5.64$10\times^{-10}$ & 0\\
       1000, T&00.0634 & 0.0248 & 0.00158 & 4.75$10\times^{-40}$ & 0\\
        \hline
    \end{tabular}
    \label{tab4}
\end{table}

\subsection{Small Parameters and variable selection}
The decreasing MSEs of the experiment highlight the advantages of using more iterations. Yet, in finite sample settings, one can make the argument to control $k$, and hence the extent of sparsity, if the emphasis is put more on controlling error for model selection. While many sparse estimators typically provide desirable asymptotic results, their performance in finite samples can sometimes suffer from inconsistency issues when small coefficients are present outside the asymptotic setting. See P\"otscher and Leeb (2009) \cite{PotscherLeeb} for detailed analyses where they revealed weaknesses of the penalized estimator in a moving parameter setting by allowing the real parameter to perturb around $0$ at an order of $n^{-1/2}$.

In tables \ref{tab5} and \ref{tab6}, we recorded the percentage of the instances where the estimator wrongly identifies $\alpha_{01} = 0.1$ and $\beta^\ast_{01} = 0.1$ as zeros, by setting a sparsity threshold of $10^{-4}$. While this threshold is arbitrary it is sufficient to depict the behaviours of the estimators. As discussed previously, we can view the number of iterations $k$ as a tuning parameter that controls the degree of sparsity, and by increasing $k$ we also risk wrongly pruning small parameters, with the percentage reaching over 60 percent for the BAR when the sample size is small. By increasing $n$ to 1000, these errors are mostly eliminated for $\alpha$ but some remain for $\beta^\ast$. 

\begin{table}[h]
    \centering
        \caption{The table shows the percentage of estimates of $\alpha_{01} = 0.1$ that are incorrectly identified as 0. The errors increase with $k$ and are particularly high for $n = 100$ for all three noise distributions (N,L,T).}
    \begin{tabular}{|c|c|c|c|c|c|c|}
        \hline
         $n,\mathcal{L}(\epsilon)$ & $k = 0$ & $k = 2$& $k = 5$ & $k = 10$& BAR\\
         \hline
        100, N&0 & 1.2 & 22.9 & 50.8 & 60.2\\
        1000, N& 0 & 0 & 0 & 0 & 0\\
        100, L&0.1 & 1.2 & 17.2 & 43.2 & 54.2\\
        1000, L& 0 & 0.3 & 0.3 & 0.3 & 0.3\\
        100, T&0 & 0.5 & 15 & 42.9 & 54.1\\
       1000, T&0 & 0.1 & 0.2 & 0.3 & 0.3\\
        \hline
    \end{tabular}
    \label{tab5}
\end{table}

\begin{table}[h]
    \centering
        \caption{The table shows the percentage of estimates of $\beta^\ast_{01} = 0.1$ that are incorrectly identified as 0. Similar to table \ref{tab5}, the errors increase with $k$ and decrease with $n$, but appear to be much more persistent. The results are shown for three noise distributions (N,L,T)}
    \begin{tabular}{|c|c|c|c|c|c|c|}
        \hline
         $n,\mathcal{L}(\epsilon)$ & $k = 0$ & $k = 2$& $k = 5$ & $k = 10$& BAR\\
         \hline
        100, N& 0 & 6.4 & 36.3 & 58.2 & 66.7\\
        1000, N& 0.2 & 4.5 & 23.1 & 35.3 & 37.6\\
        100, L&0.1 & 4.9 & 31.8 & 54.8 & 63.7\\
        1000, L& 0.2 & 4.4 & 26.4 & 38.2 & 39.9\\
       100, T&0 & 5.6 & 33.3 & 55.7 & 65\\
       1000, T&0.1 & 4 & 25.1 & 37.1 & 39.2\\
        \hline
    \end{tabular}
    \label{tab6}
\end{table}

Table \ref{tab7} tabulates the false negative and positive rates for the AR and BAR estimators, providing a fairer judgment of performance. Here, a false negative is wrongly estimating a non-zero estimate to be zero, and a false positive is the failure to identify a zero parameter. As we did previously, we set the sparsity threshold equal to $10^{-4}$. As the adaptive ridge with $k = 0$ is simply the regular ridge, it cannot identify sparse signals, as shown by the near $0$ and $100$ false negative and positive rates. Then, the false negative rates increase with $k$ and the false positive rates fall with the number of iterations. These rates imply that no estimator is completely "better" than another, and the optimal choice of $k$ will depend on the user's preference for balancing errors. 
\begin{table}[h]
    \centering
        \caption{The false positive (FP) and negative (FN) error rates for the AR and BAR under 6 different scenarios ($n = 100, 1000$, noise = N,L,T). Iterating more leads to a higher degree of sparsity, and can be reflected in the higher false negative rates but lower false positive rates.}
    \begin{tabular}{|c|c|c|c|c|c|c}
        \hline
         $n,\mathcal{L}(\epsilon)$& Iterations& FN ($\alpha,\beta^\ast$) & FP ($\alpha,\beta^\ast$)\\
         \hline
        \multirow{5}{*}{$100, N(0,1)$}
        &$k$ = 0 & 0 , 0.01 & 99.99 , 99.98\\
        &$k$ = 2 &0.38 , 3.31 & 98.82 , 95.58 \\
        &$k$ = 5 & 4.48 , 17.6 & 72.16 , 68.86\\
        &$k$ = 10 & 8.61 , 31.19 & 39.48 , 44.18\\
        &BAR & 10.23 , 39.82 & 28.8 , 33.19\\
        \hline
        \multirow{4}{*}{$1000, N(0,1)$}
        &$k$ = 0 & 0.01 , 0.02 & 100 , 99.93\\
        &$k$ = 2 & 0.07 , 1.01 & 95 , 91.72 \\
        &$k$ = 5 & 0.07 , 3.91 & 38.83 , 52.87\\
        &$k$ = 10 & 0.07 , 5.59 & 17.64 , 38.28\\
        &BAR &0.07 , 5.85 & 13.51 , 36.75\\
        \hline
        \multirow{4}{*}{$100, \text{Laplace}(0,1)$}
        &$k$ = 0 &0.01 , 0.02 & 100 , 99.97\\
        &$k$ = 2 &0.4 , 3.52 & 99.21 , 95.27\\
        &$k$ = 5 & 3.96 , 18.85 & 79.35 , 70.24\\
        &$k$ = 10 &8.89 , 33.34 & 48.99 , 46.54\\
        &BAR & 11.49 , 42.3 & 37.37 , 34.7\\
        \hline
        \multirow{4}{*}{$1000, \text{Laplace}(0,1)$}
        &$k$ = 0 &0.02 , 0.06 & 100 , 99.95\\
        &$k$ = 2 & 0.05 , 1.01 & 96.86 , 92.62 \\
        &$k$ = 5 & 0.05 , 4.65 & 51.41 , 58.39\\
        &$k$ = 10 &0.05 , 6.43 & 29.15 , 43.5\\
        &BAR &0.05 , 6.64 & 24.51 , 41.44\\
        \hline
        \multirow{4}{*}{$100, t(3)$}
        &$k$ = 0 &0 , 0.01 & 100 , 99.99 \\
        &$k$ = 2 &0.44 , 3.13 & 99.17 , 95.37 \\
        &$k$ = 5 & 3.97 , 18.76 & 81.29 , 69.95\\
        &$k$ = 10 &9.7 , 32.65 & 51.89 , 46.15\\
        &BAR & 12.56 , 42.6 & 39.78 , 34.47\\
        \hline
        \multirow{4}{*}{$1000, t(3)$}
        &$k$ = 0 & 0.01 , 0.07 & 100 , 99.94 \\
        &$k$ = 2 & 0.03 , 1.04 & 97.37 , 92.05\\
        &$k$ = 5 & 0.04 , 4.52 & 58.11 , 55.99\\
        &$k$ = 10 & 0.05 , 6.18 & 37 , 41.09\\
        &BAR & 0.05 , 6.47 & 32.02 , 38.76\\
        \hline
    \end{tabular}
    \label{tab7}
\end{table}

\subsection{Experimental run: possible relaxation of negative moments condition}
\label{hm:ss_moment}
We conclude the simulation section by showing some numerical evidence which suggests possible relaxation of the annoying condition
\begin{equation}\label{epsilonharsh}
    \mathbb{E}\left[|\epsilon_i|^{-(1+\tau)}\right] < \infty \quad \text{for some} \quad \tau > 0,
\end{equation}
that we needed to prove the asymptotic normality of $\bast$ in Theorem \ref{thm1}.
While providing useful theoretical guarantees, this assumption is restrictive and excludes common noise distributions such as the normal. In the hopes of generalizing and incorporating a wider class of noise distributions, we numerically show that this condition on $\epsilon$ can be lifted. Figure \ref{fig2} shows strong normality evidence with the QQ-plots of $\tilde\beta_{01}^{\ast(0)}$ under the three aforementioned noise distributions, as they all fail to satisfy \eqref{epsilonharsh}. This generalization may then lead to the relaxation of the conditions in both Theorems \ref{thm2} and \ref{thm3}, providing more unified results for many commonly considered noises. 

\begin{figure}
    \centering
    \includegraphics[width=0.32\linewidth]{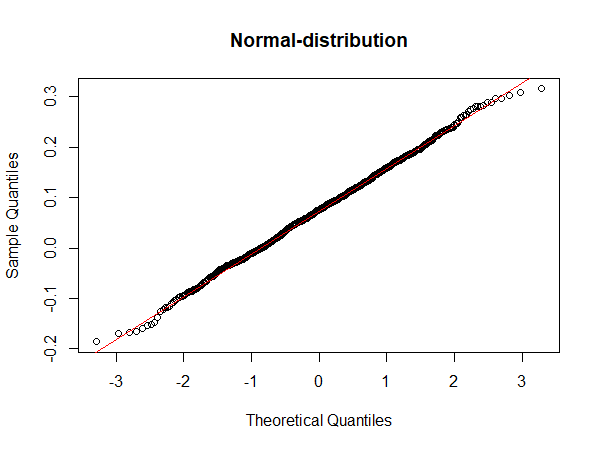}
    \includegraphics[width=0.32\linewidth]{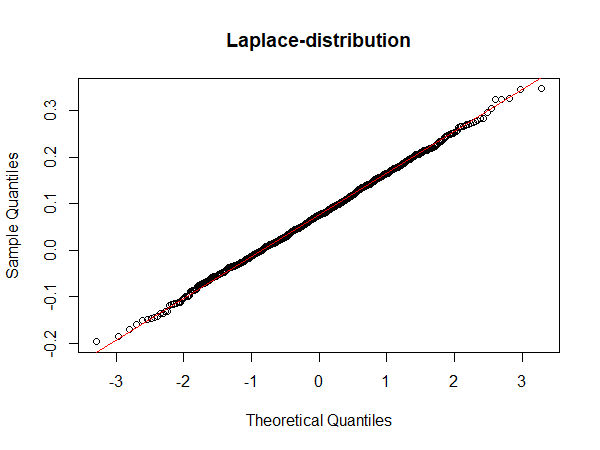}
    \includegraphics[width=0.32\linewidth]{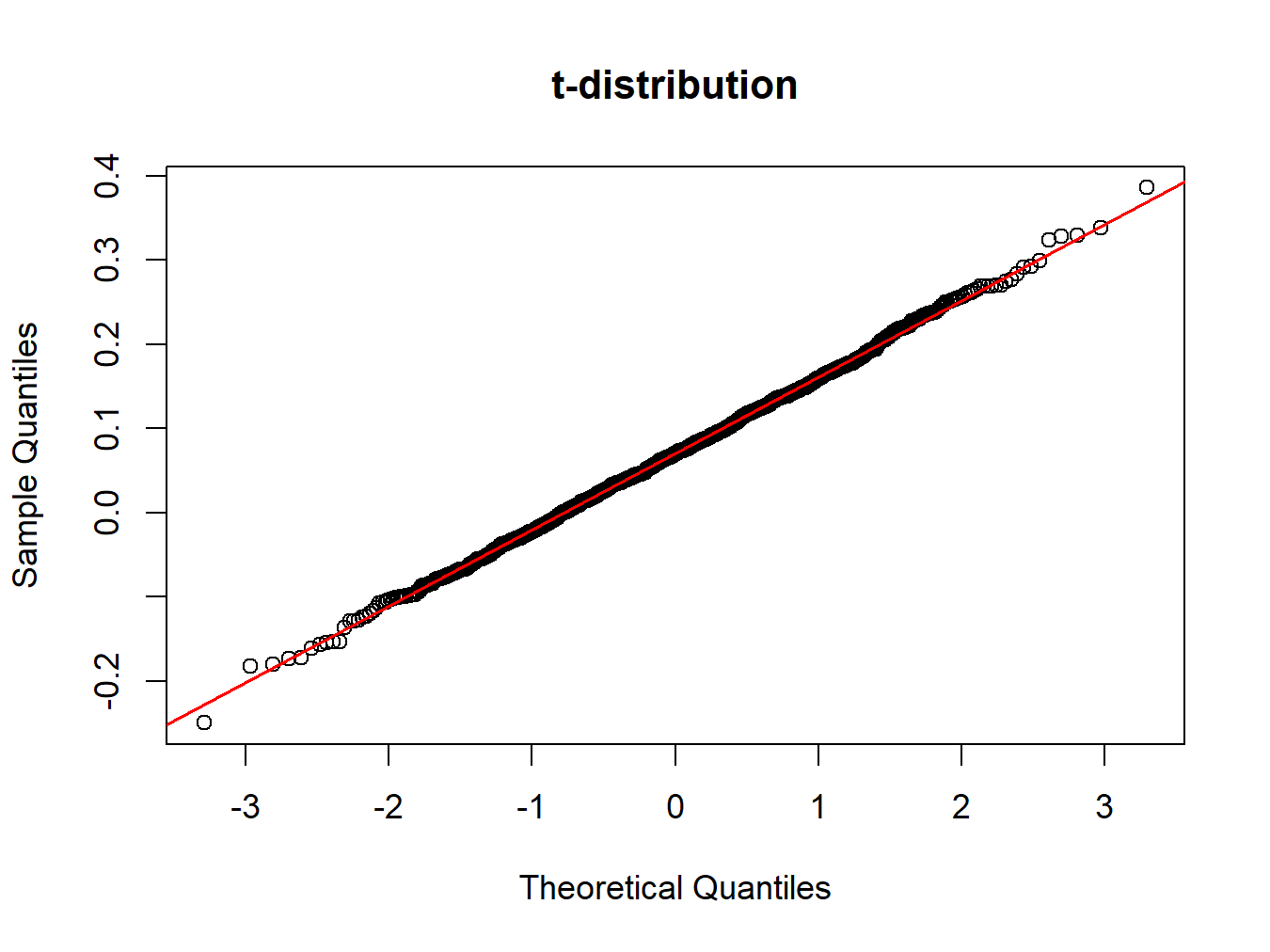}
    \caption{Figure shows numerical evidence of normality for the initial ridge estimates of $\beta^\ast$, as we show the QQ-plots of the first component of $\tilde\beta_{01}^{(0)}$ under three different noise distributions. These plots are generated for $n = 1000$ over 1000 independent trials.}
    \label{fig2}
\end{figure}

\section{Data Analysis}\label{sec5}
\subsection{Boston Housing Prices}\label{sec5.1}
We analysed the Boston housing dataset, which contains $n = 506$ samples for median house prices (Harrison and Rubinfeld, 1978) \cite{BostonHousing}. We set $X = Z$ as our design matrices, which consists of $p = q = 13$ covariates. We standardize $X$ and $Z$ so that their columns have zero mean and unit variance. We also add a column of $1$'s to $Z$ as we did in previous sections. We randomly split the data into training and testing sets of sizes 400 and 106. We assume model \eqref{eq1.1} and compare the results of the AR for $k = 0,2,5,10$ and the BAR using mean squared prediction errors (MSPE).

The tuning matrices are set as a multiple of the identity as in Section \ref{sec4}, then we first did a grid search for the optimal value of $\psi$ using 5-fold cross-validation over $\psi \in \{10^k; k = -5,-4.9,...,5\}$, where we obtained the optimal value to be $\psi = 50.1$. We then set $\omega = 1$ and then perform a further grid search for the AR tuning parameters on $(\lambda,\gamma) \in \{(10^i,10^j); i,j = -2,-1.8,...,2)\}$ using the AR ($k = 10$), where we obtained $(\lambda,\gamma) = (0.01,25.1)$ for $k = 2$, $(\lambda,\gamma) = (2.51,6.31)$ for $k = 5$, $(\lambda,\gamma) = (0.158,0.0158)$ for $k = 10$ and $(\lambda,\gamma) = (1,0.01)$ for the BAR. Here, we used the median squared prediction error instead of the MSPE to evaluate CV errors as we found it to produce better and more consistent tuning parameters and was less swayed by extreme data values across many different training sets, which resulted in better prediction accuracy for all estimators.  

The results are shown in table \ref{tab8}, where we observe the clear gap between $k = 0$ and the other estimators. Referring to the simulation results in Section \ref{sec4}, this is likely due to the lack of modelling heteroscedasticity in the case of $k = 0$. Indeed, the number of significant components in the variance component identified by the AR procedure is shown to be 8 or 9, showing the presence of relationships between variance and the covariates. Overall, despite the wide range of hyperparameters used, the resulting models and performances were nearly identical, and all showed significant improvement compared to a model without heteroscedasticity.

\begin{table}
    \centering
    \caption{The mean squared prediction error of each adaptive ridge predictor, and their number of chosen variables when fitted to the Boston Housing dataset (Harrison and Rubinfeld, 1978) \cite{BostonHousing}. Heteroscedasticity was detected for more than the covariates and its consideration subsequently led to reduced prediction error.}
    \label{tab8}
    \begin{tabular}{|c|c|c|c|c|c|}
     \hline
                &   $k = 0$   &  $k = 2$  &   $k = 5$  &   $k = 10$  & BAR\\
            \hline
Selected Predictors ($\alpha$) & 5 & 6 & 6 & 6 & 6 \\
Selected Predictors ($\beta$) & n/a & 7 & 8 & 8 & 8 \\
\hline
MSPE & 26.15 & 24.8 & 24.94 & 24.72 & 24.86\\
\hline
    \end{tabular}
\end{table}

\subsection{Electricity consumption prediction}
We also fitted electricity consumption data obtained from Tokyo Electric Power Company Holdings, Incorporated (TEPCO)\footnote{("https://www4.tepco.co.jp/en/forecast/html/download-e.html")} for 2018 and 2019. The dataset contains records of the highest hourly consumption, contributing to a total of $(2 \times 365) \times 24 = 17520$ entries. We train with the 2018 set and test with the 2019 data. 

The design matrices $X=Z$ contain 42 features, including days of the week, hours of the day, temperature, public holidays, weekend public holidays, and autoregressive terms. $X$ and $Z$ are standardized so that each column has mean 0 and unit variance, and the response $Y$ is centralized to have mean 0. We also augment $Z$ with a column of $1$'s as in the previous sections. The same method of grid search for hyperparameters from Section \ref{sec5.1} was used, but contrary to using the median, using the MSPE in the cross-validation step yielded more consistent and well-performing models. In the end, the hyperparameters were set to be $\psi =  0.5$, $\omega = 1$, $\lambda = 1000$ and $\gamma = 10000$ for all 3 values of $k$ and subsequently the BAR. 

We compared the mean squared prediction error (MSPE) for each of the AR estimators and the values are shown in table \ref{tab9} along with the number of selected predictors. Here, the AR procedure eliminated all heteroscedastic terms, and the lack of heteroscedasticity led to the lowest MSPE obtained for the regular ridge estimate. However, if we compare the AR for $k = 10$ and the BAR with the regular ridge ($k = 0$), performance was nearly identical but the AR utilized 5 fewer covariates as the number of selected $\alpha$ variables decreased from 19 to 14. We wish to use this data example to motivate future work to extend current results to time series models, enhancing its applicability. 
\begin{table}
    \centering
    \caption{The mean squared prediction error and number of chosen variables for each AR/BAR estimator. The noise was determined to be homogeneous as the performance stayed stable as we iterated.}
    \label{tab9}
    \begin{tabular}{|c|c|c|c|c|c|}
     \hline
    &   $k = 0$   &  $k = 2$  &   $k = 5$  &   $k = 10$  & BAR\\
    \hline
Selected Predictors ($\alpha$) & 19 & 19 & 14 & 14 & 14 \\
Selected Predictors ($\beta$) & n/a & 10 & 0 & 0 & 0 \\
\hline
MSPE & 11447 & 12046 & 11467 & 11454 & 11454\\
\hline
    \end{tabular}

\end{table}

\section{Discussion and Concluding Remarks}\label{sec6}
One area for future work is to extend the current theory to include general positive definite matrices as tuning matrices, as we believe similar results should arise from generic eigenvalue arguments. On this avenue, we wish also to introduce more polished methods for selecting hyperparameters. 

Motivated by the numerical evidence given in Section \ref{hm:ss_moment}, 
the assumptions on $\epsilon$ can likely be weakened so that asymptotic distributional results for $\bast$ hold under a wider class of errors. Further, to strengthen the result given in Theorem \ref{thm2}, one can attempt to construct approximate confidence sets for the iterated estimators of $k \geq 1$. 

Finally, with the success of ordinary ridge estimators in high-dimensional settings, it will be interesting to see the AR as an extension of high-dimensional ridge regression. We are also optimistic that the theory of the AR scheme can be generalized to other frameworks such as time series and stochastic differential equations. 

\subsection*{Acknowledgement}
The first author (KLKH) would like to thank JGMI of Kyushu University for their support. Also, this work was partially supported by JST CREST Grant Number JPMJCR2115 and JSPS KAKENHI Grant Number 22H01139, Japan (HM).

\newpage
\appendix\normalsize
\section{Proof of Theorem \ref{thm1}}\label{secA}
We begin by stating the matrix inversion formula (also known as the first-order resolvent formula or Sherman-Morrison-Woodbury formula), used as in Dai et al.(2018) \cite{BAR2018}:
\begin{equation}\label{eq7.1}
(X^T X + \Psi_n)^{-1} = (X^T X)^{-1} - (X^T X)^{-1} \Psi_n (X^T X + \Psi_n)^{-1}.
\end{equation}
Using \eqref{eq7.1}, we write
\begin{equation}\label{eq7.2}
\begin{split}
    &\sqrt{n}(\tilde\alpha^{(0)}_n-\alpha_0)\\ =& \sqrt{n}\left[(X^T X + \Psi_n)^{-1}X^T Y-\alpha_0\right] \\
    =& \sqrt{n}\left[\left\{(X^T X)^{-1} - (X^T X)^{-1} \Psi_n (X^T X + \Psi_n)^{-1}\right\}X^T (X\alpha_0+D_n(\beta_0)\epsilon)-\alpha_0\right] \\
    =& \left(\frac{X^T X}{n}\right)^{-1}\frac{1}{\sqrt{n}}X^T D_n(\beta_0)\epsilon - \left(\frac{X^T X}{n}\right)^{-1} \frac{\Psi_n}{\sqrt{n}} \left(\frac{X^T X}{n} + \frac{\Psi_n}{n}\right)^{-1}\frac{X^T Y}{n}.
\end{split}
\end{equation}
Let us first consider term $\left(\frac{X^T X}{n}\right)^{-1}\frac{1}{\sqrt{n}}X^T D_n(\beta_0)\epsilon$, where we can apply the Lindeberg Central Limit Theorem. By writing
\begin{align*}
    \frac{1}{\sqrt{n}}X^T D_n(\beta_0)\epsilon = \frac{1}{\sqrt{n}}\sumi{e^{\frac{1}{2}Z_i^T\beta_0}X_i\epsilon_i},
\end{align*}
the mean is computed as$\frac{1}{\sqrt{n}}\sumi{\mathbb{E}\left[e^{\frac{1}{2}Z_i^T\beta_0}X_i\epsilon_i\right]} = 0$, where $0$ denotes the $p$-dimensional $0$-vector here. Also, the covariance
\begin{align*}
    \frac{1}{n}\sumi{\cov[e^{\frac{1}{2}Z_i^T\beta_0}X_i\epsilon_i}] = \frac{1}{n}\sumi{e^{Z_i^T\beta_0}X_i^TX_i\mathbb{E}\left[\epsilon_i^2\right]} 
    = \frac{X^T D^2_n(\beta_0)X}{n}
    \to \Sigma^X_+
\end{align*}
under Assumption \ref{(A1)}. Finally, to check for Lindeberg's condition, we let $m_n := \max_{1\leq i \leq n}[\frac{e^{Z_i^T\beta_0}X_iX_i^T}{n}]$ and note that $m_n \to 0$ by Assumption \ref{(A2)}. Thus for any $\delta > 0$,
\begin{align*}
    &\frac{1}{n}\sumi{\mathbb{E}\left[\Vert e^{\frac{1}{2}Z_i^T\beta_0}X_i\epsilon_i \Vert^2;\Vert e^{\frac{1}{2}Z_i^T\beta_0}X_i\epsilon_i \Vert^2 \geq n\delta\right]}\\
    &= \sumi{\frac{e^{Z_i^T\beta_0}X_iX_i^T}{n} \mathbb{E}\left[\epsilon_i^2 ; \epsilon_i^2 \geq \dfrac{\delta}{\frac{e^{Z_i^T\beta_0}X_iX_i^T}{n}}\right]}\\
    &\leq \mathbb{E}\left[\epsilon_1^2 ; \epsilon_1^2 > \frac{\delta}{m_n}\right] \sumi{\frac{e^{Z_i^T\beta_0}X_iX_i^T}{n}}\\
    &= \tr\left({\frac{X^T D_n^2(\beta_0)X}{n}}\right)\mathbb{E}\left[\epsilon_1^2 ; \epsilon_1^2 > \frac{\delta}{m_n}\right]\\
    \to& 0,
\end{align*}
where the final limit follows from $\frac{X^T D_n^{2}(\beta_0)X}{n} \to \Sigma_+^X$ and applying Lebesgue's Dominated Convergence Theorem on the sequence $\mathbb{E}\left[\epsilon_1^2 ; \epsilon_1^2 > \frac{\delta}{m_n}\right]$.
Thus the first term of \eqref{eq7.2} converges in law: 
\begin{equation}\label{eq7.3}
    \left(\frac{X^T X}{n}\right)^{-1}\frac{1}{\sqrt{n}}X^T D_n(\beta_0)\epsilon \cil N\left(0,(\Sigma^X)^{-1}\Sigma^X_+(\Sigma^X)^{-1}\right).
\end{equation}
As for the other term of the sum in \eqref{eq7.2}, we rewrite it as 
\begin{align*}
    &\left(\frac{X^T X}{n}\right)^{-1} \frac{\Psi_n}{\sqrt{n}} \left(\frac{X^T X}{n} + \frac{\Psi_n}{n}\right)^{-1}\frac{X^T Y}{n}\\ 
    =& \left(\frac{X^T X}{n}\right)^{-1} \frac{\Psi_n}{\sqrt{n}} \left(\frac{X^T X}{n} + \frac{\Psi_n}{n}\right)^{-1}\frac{1}{n}X^T\left(X\alpha_0 + D_n(\beta_0)\epsilon\right),
\end{align*}
which we further evaluate as a sum of two terms. By \eqref{eq7.3} and $\Psi_n/\sqrt{n} \to \Psi_0$ ($\Psi_n/n \to 0$), it follows that 
\[\left(\frac{X^T X}{n}\right)^{-1} \frac{\Psi_n}{\sqrt{n}} \left(\frac{X^T X}{n} + \frac{\Psi_n}{n}\right)^{-1}\frac{X^T X}{n}\alpha_0 \to (\Sigma^X)^{-1}\Psi_0\alpha_0,\]
and
\[
\left(\frac{X^T X}{n}\right)^{-1} \frac{\Psi_n}{\sqrt{n}} \left(\frac{X^T X}{n} + \frac{\Psi_n}{n}\right)^{-1}\frac{1}{n}X^T D_n(\beta_0)\epsilon \cip 0. 
\]
Combining these results with \eqref{eq7.3} show asymptotic normality as stated in \eqref{eq3.3}:
\begin{align*}
    \sqrt{n}(\tilde\alpha^{(0)}_n-\alpha_0) \cil N \left(-(\Sigma^X)^{-1}\Psi_0\alpha_0, (\Sigma^X)^{-1}\Sigma^X_+(\Sigma^X)^{-1}\right).
\end{align*}
We now proceed to the asymptotics of $\bastinit$. First note that 
\begin{align*}
    L_n(\alpha_0) &= Z\beta_0 + \log(\epsilon^2)\\
    &= \Zast \basttrue + \log(\epsilon^2) - c_0.
\end{align*}
 We denote $E_n := \log(\epsilon^2) - c_0$ and $R_n : = L_n( \tilde\alpha_n^{(0)}) - L_n(\alpha_0)$, so that for $c \in [0,\frac{1}{2})$,
\begin{equation}\label{eq7.4}
\begin{split}
    &n^c\left(\bastinit - \basttrue\right)\\ =& n^c\left((Z^{\ast T}\Zast + \Omega_n^\ast)^{-1}Z^{\ast T}L_n(\tilde\alpha_n^{(0)}) - \basttrue\right) \\
    =& n^c\left[\left(\frac{Z^{\ast T}\Zast }{n}+\frac{\Omega_n^\ast}{n}\right)^{-1}\frac{Z^{\ast T}\Zast \beta_0^\ast}{n}- \beta_0^\ast\right] + \left(\frac{Z^{\ast T}\Zast }{n}+\frac{\Omega_n^\ast}{n}\right)^{-1}\frac{Z^{\ast T}E_n}{n^{1-c}}\\ 
    &{}\qquad + \left(\frac{Z^{\ast T}\Zast }{n}+\frac{\Omega_n^\ast}{n}\right)^{-1}\frac{Z^{\ast T}R_n}{n^{1-c}}.
\end{split}
\end{equation}
By \eqref{eq7.1}, the first term of the sum vanishes as $n \to \infty$ for $c \leq \frac{1}{2}$ under $\frac{\Omega_n^\ast}{\sqrt{n}} \to 0$:
\begin{align*}
    &n^c\left[\left(\frac{Z^{\ast T}\Zast }{n}+\frac{\Omega_n^\ast}{n}\right)^{-1}\frac{Z^{\ast T}\Zast \beta_0^\ast}{n}- \beta_0^\ast\right]\\
    &= n^c\bigg[\left(\frac{Z^{\ast T}\Zast }{n}\right)^{-1}\frac{Z^{\ast T}\Zast \beta_0^\ast}{n} 
    \nn\\
    &{}\qquad 
    - \left(\frac{Z^{\ast T}\Zast }{n}\right)^{-1}\frac{\Omega_n^\ast}{n}\left(\frac{Z^{\ast T}\Zast }{n}+\frac{\Omega_n^\ast}{n}\right)^{-1}\frac{Z^{\ast T}\Zast \beta_0^\ast}{n} - \beta_0^\ast\bigg]\\
    &= -\left(\frac{Z^{\ast T}\Zast }{n}\right)^{-1}\frac{\Omega_n^\ast}{n^{1-c}}\left(\frac{Z^{\ast T}\Zast }{n}+\frac{\Omega_n^\ast}{n}\right)^{-1}\frac{Z^{\ast T}\Zast \beta_0^\ast}{n}\\
    \to & \text{ }0.
\end{align*}
Then, if we write $\frac{Z^{\ast T}E_n}{\sqrt{n}} = \sumi{\frac{Z _i^{\ast}E_{ni}}{\sqrt{n}}}$ and use the Lindeberg Central limit theorem analogously, we obtain under Assumption \ref{(A1)} and $\frac{\Omega_n^\ast}{\sqrt{n}} \to 0$ that
\begin{equation}\label{eq7.5}
    \left(\frac{Z^{\ast T}\Zast }{n}\right)^{-1}\frac{Z^{\ast T}E_n}{\sqrt{n}} \cil N\left(0, (\Sigma^{\Zast })^{-1}\var[\log(\epsilon_1^2)]\right),
\end{equation}
which also proves that the second term of \eqref{eq7.4} is $\smallO_p(1)$ if $c < 1/2$. 

We will first prove \eqref{eq3.4}, and to this end, we are going to show the tightness of the third term of $\eqref{eq7.4}$: 
\begin{align*}
    \left(\frac{Z^{\ast T}\Zast}{n}+\frac{\Omega_n^\ast}{n}\right)^{-1}\frac{Z^{\ast T}R_n}{n^{1-c}} = \mathcal{O}_p(1).
\end{align*}
By Assumption \ref{(A1)}, it is sufficient to show 
\begin{align*}
    \frac{Z^{\ast T}R_n}{n^{1-c}} = \mathcal{O}_p(1).
\end{align*}
We denote $\tilde u^{(0)}_n = \sqrt{n}(\tilde\alpha_n^{(0)} - \alpha_0)$, $\tilde  r_{ni}^{(0)} = \dfrac{X_i^T \tilde u^{(0)}_n}{e^{\frac{1}{2}Z_i^T\beta_0}\epsilon_i}$,  $\tilde \psi^{(0)}_{ni} = \log\bigg|1 - \frac{1}{\sqrt{n}}\tilde  r_{ni}^{(0)}\bigg|$, and $I(A)$ as the indicator function of the event $A$.
Then,
\begin{equation*}
    \begin{split}
    \frac{Z^{\ast T}R_n}{n} &= \frac{2}{n} \sumi{\Zast _i\tilde \psi^{(0)}_{ni}}\\
    &=  \frac{2}{n} \sumi{\Zast _i\tilde \psi^{(0)}_{ni}\left(I\bigg(\bigg|\frac{\tilde  r_{ni}^{(0)}}{\sqrt{n}}\bigg| \leq  \frac{1}{2}\bigg) + I\bigg(\bigg|\frac{\tilde  r_{ni}^{(0)}}{\sqrt{n}}\bigg| > \frac{1}{2}\bigg) \right)}\\
    &=: M_{1,n} + M_{2,n}.
    \end{split}
\end{equation*}
We first treat $M_{1,n}$. Since for all $s \in [0,1)$ and $u \in \mathbb{R}$, there exists some constant $C_s$ so that
\begin{align*}
    \bigg|\log|1-u|\bigg|I\bigg(u \leq \frac{1}{2}\bigg) \leq C_s|u|^s.
\end{align*}
Using this with $u = \frac{1}{\sqrt{n}}\tilde  r_{ni}^{(0)}$ we get by Assumption \ref{(A2)}
\begin{equation}\label{eqA7}
     \left\Vert \frac{2}{n}\sumi{\Zast_i\tilde \psi^{(0)}_{ni}I\bigg(\bigg|\frac{\tilde  r_{ni}^{(0)}}{\sqrt{n}}\bigg| \leq  \frac{1}{2}\bigg)} \right \Vert \leq \frac{2C_s \Vert Z^\ast \Vert}{n}\sumi{\bigg|\frac{\tilde  r_{ni}^{(0)}}{\sqrt{n}}\bigg|^s} \lesssim \frac{1}{n}\sumi{\bigg|\frac{\tilde  r_{ni}^{(0)}}{\sqrt{n}}\bigg|^s}.
\end{equation}
The following lemma is required next.
\begin{lem}\label{lem1}
    There exists some $N$ such that for all $s' > 1$, \[ \sup_{n > N} \mathbb{E}\left[ \left \Vert \sqrt{n}(\tilde\alpha_n^{(0)} - \alpha_0) \right \Vert^{s}\right] < \infty.\]
\end{lem}

\begin{proof}
    Splitting $\sqrt{n}\left(\tilde\alpha_n^{(0)} - \alpha_0 \right)$ as \eqref{eq7.2}, we have for $s' > 1$,
    \begin{align*}
    &\mathbb{E}\left[\left\Vert \sqrt{n}\left(\tilde\alpha_n^{(0)} - \alpha_0 \right)\right\Vert^{s'}\right]\\
    &\leq \mathbb{E}\left[
    \left\Vert\left(\frac{X^T X}{n}\right)^{-1}\frac{1}{\sqrt{n}}X^T D_n(\beta_0)\epsilon\right\Vert^{s'}\right]\\
    & \quad+ \mathbb{E}\left[\left\Vert \left(\frac{X^T X}{n}\right)^{-1} \frac{\Psi_n}{\sqrt{n}} \left(\frac{X^T X}{n} + \frac{\Psi_n}{n}\right)^{-1}\frac{1}{n}X^T D_n(\beta_0)\epsilon\right\Vert^{s'}\right]\\
    &\quad+ \mathbb{E}\left[\left\Vert \left(\frac{X^T X}{n}\right)^{-1} \frac{\Psi_n}{\sqrt{n}} \left(\frac{X^T X}{n} + \frac{\Psi_n}{n}\right)^{-1}\frac{X^T X}{n}\alpha_0\right\Vert^{s'}\right].
    \end{align*}
The final term of the sum is non-random and so is finite for all $s'$. Then, using Assumptions \ref{(A1)} and \ref{(A2)}, we can also bound the first and second terms:
\begin{align*}
    \mathbb{E}\left[
    \left\Vert\left(\frac{X^T X}{n}\right)^{-1}\frac{1}{\sqrt{n}}X^T D_n(\beta_0)\epsilon\right\Vert^{s'}\right] \lesssim \mathbb{E}\left[\left\Vert \frac{1}{\sqrt{n}} \sumi \epsilon_i \right\Vert^{s'}\right]
\end{align*}
and 
\begin{align*}
    \mathbb{E}\left[\left\Vert \left(\frac{X^T X}{n}\right)^{-1} \frac{\Psi_n}{\sqrt{n}} \left(\frac{X^T X}{n} + \frac{\Psi_n}{n}\right)^{-1}\frac{1}{n}X^T D_n(\beta_0)\epsilon\right\Vert^{s'}\right] \lesssim \mathbb{E}\left[\left\Vert \frac{1}{n}\sumi \epsilon_i \right\Vert^{s'}\right].
\end{align*}
Therefore, taking the supremum yields
\begin{align*}
    \sup_{n>N} \mathbb{E}\left[\left\Vert \sqrt{n}\left(\tilde\alpha_n^{(0)} - \alpha_0 \right)\right\Vert^{s'}\right] &\lesssim \sup_{n>N} \mathbb{E}\left[\left\Vert \frac{1}{\sqrt{n}}\sumi \epsilon_i \right\Vert^{s'}\right].
\end{align*}
Finally, $\mathbb{E}\left[\Vert\epsilon_1\Vert^{s'}\right] < \infty$ (Assumption \ref{(A3)}) suffices for $\sup_{n>N} \mathbb{E}\left[\left\Vert \frac{1}{\sqrt{n}}\sumi \epsilon_i \right\Vert^{s'}\right] < \infty$ (Bahr(1965) \cite{VonBahr1965}, DasGupta (2008) \cite{DasGupta2008}).
\end{proof}

We will now estimate the upper bound in \eqref{eqA7} using the Markov and H\"older inequalities: under Assumptions \ref{(A2)} and \ref{(A3)}, for any positive constant $M$:
\begin{equation}\label{eq7.7}
    \begin{split}
        \sup_n \mathbb{P}\left[\frac{1}{n}\sumi{\bigg|\frac{\tilde  r_{ni}^{(0)}}{\sqrt{n}}\bigg|^s} > M n^{-s/2}\right] & \leq \sup_n \mathbb{P}\left[\frac{1}{n}\sumi \frac{\Vert \tilde u_n^{(0)}\Vert^s}{|\epsilon_i|^s} \gtrsim M \right]\\
        &\lesssim \sup_n \frac{1}{n}\sumi \mathbb{E}\left[\frac{\Vert \tilde u_n^{(0)}\Vert^s}{|\epsilon_i|^s}\right]\\
        &\lesssim \frac{1}{M}\sup_n \mathbb{E}\left[\Vert \tilde u_n^{(0)}\Vert^{sa}\right]^{\frac{1}{a}} \mathbb{E}\left[|\epsilon_1|^{-sb}\right]^{\frac{1}{b}}\\
        & \to 0        
    \end{split}
\end{equation}
as $M \to \infty$ by Lemma \ref{lem1}, provided that we choose $a$ and $b$ such that $a^{-1} + b^{-1} = 1$ and $sb < 1$. Therefore,
\begin{equation}\label{eq7.8}
    M_{1,n} = \frac{2}{n}\sumi{\Zast _i\tilde \psi^{(0)}_{ni}I\bigg(\bigg|\frac{\tilde  r_{ni}^{(0)}}{\sqrt{n}}\bigg| \leq  \frac{1}{2}\bigg)} = \mathcal{O}_p(n^{-\frac{s}{2}}).
\end{equation}
We now turn our attention to $M_{2,n}$. Taking expectations, we have 
\begin{align*}
   & \mathbb{E}\left[ \left \Vert \frac{2}{n}\sumi{\Zast _i\tilde \psi^{(0)}_{ni}I\bigg(\bigg|\frac{\tilde  r_{ni}^{(0)}}{\sqrt{n}}\bigg| >  \frac{1}{2}\bigg)}\right \Vert \right]
   \nn\\
&\lesssim \frac{1}{n}\sumi{\mathbb{E}\left[|\tilde \psi^{(0)}_{ni}|I\bigg(\bigg|\frac{\tilde  r_{ni}^{(0)}}{\sqrt{n}}\bigg| >  \frac{1}{2}\bigg)\right]}\\
    &\leq \frac{1}{n} \sumi{\mathbb{E}\left[|\tilde \psi^{(0)}_{ni}|^a\right]^{\frac{1}{a}}\mathbb{P}\left[ \left|\frac{\tilde  r_{ni}^{(0)}}{\sqrt{n}} \right | >  \frac{1}{2}\right]^{\frac{1}{b}}}\\
    &\leq \frac{1}{n}\left(\sup_n \sup_i \mathbb{E}\left[|\tilde \psi^{(0)}_{ni}|^a\right]^{\frac{1}{a}} \right)\sumi{\mathbb{P}\left[ \left|\frac{\tilde  r_{ni}^{(0)}}{\sqrt{n}} \right | >  \frac{1}{2}\right]^{\frac{1}{b}}},
\end{align*}
where we chose the exponents $a, b > 1$  $(1/a + 1/b = 1)$ independently of that in \eqref{eq7.7}.
We bound $\frac{1}{n}\sumi{\mathbb{P}\left[\left|\frac{\tilde  r_{ni}^{(0)}}{\sqrt{n}}\right| >  \frac{1}{2}\right]^{\frac{1}{b}}}$ similar to \eqref{eq7.7}. For any $t \in (s,1)$ there is some constant $C_t$ such that
\begin{equation}\label{eq7.9}
    \begin{split}
        \frac{1}{n}\sumi{\mathbb{P}\left[\left| \frac{\tilde r_{ni}^{(0)}}{\sqrt{n}}\right| >  \frac{1}{2}\right]^{\frac{1}{b}}} & = \frac{1}{n}\sumi{\mathbb{P}\left[\left| \tilde  r_{ni}^{(0)}\right|^t >  2^{-t}n^{\frac{t}{2}} \right]^{\frac{1}{b}}}\\
        & \leq \frac{1}{n}\sumi{\mathbb{P}\left[\frac{\Vert \tilde u_n^{(0)}\Vert^t}{|\epsilon_i|^t} >  C_t n^{\frac{t}{2}}\right]^{\frac{1}{b}}}\\
        &\lesssim \frac{1}{n}\sumi{n^{-\frac{t}{2b}} \mathbb{E}\left[ \frac{\Vert \tilde u_n^{(0)}\Vert^t}{|\epsilon_i|^t} \right]^{\frac{1}{b}}}\\
        & \lesssim n^{-\frac{t}{2b}}.
    \end{split}
\end{equation}
By choosing $b$ and $t$ so that $b>1$ and $\frac{t}{b} = s$, we get consistent results with \eqref{eq7.8}. Finally, for any $a' \in (0,1)$,
\begin{align*}
    |\tilde  \psi_{ni}^{(0)}|^a &= \left|\log \bigg|\frac{e^{\frac{1}{2}Z_i^T\beta_0}\epsilon_i - \frac{1}{\sqrt{n}}X_i^T \tilde u_n^{(0)}}{e^{\frac{1}{2}Z_i^T\beta_0}\epsilon_i}\bigg|^a \right|\\
    &\leq \bigg(\log\bigg|e^{\frac{1}{2}Z_i^T\beta_0}\epsilon_i -  \frac{1}{\sqrt{n}}X_i^T\tilde u_n^{(0)}\bigg| + \log|e^{\frac{1}{2}Z_i^T\beta_0}\epsilon_i|\bigg)^a\\
    &\lesssim \bigg(\log\bigg|e^{\frac{1}{2}Z_i^T\beta_0}\epsilon_i -  \frac{1}{\sqrt{n}}X_i^T\tilde u_n^{(0)}\bigg|\bigg)^a + \bigg(\log|e^{\frac{1}{2}Z_i^T\beta_0}\epsilon_i|\bigg)^a \\
    &\lesssim \bigg|e^{\frac{1}{2}Z_i^T\beta_0}\epsilon_i -  \frac{1}{\sqrt{n}}X_i^T\tilde u_n^{(0)}\bigg|^{a'} + \bigg|e^{\frac{1}{2}Z_i^T\beta_0}\epsilon_i -  \frac{1}{\sqrt{n}}X_i^T\tilde u_n^{(0)}\bigg|^{-a'}\\
    &\quad + \bigg|e^{\frac{1}{2}Z_i^T\beta_0}\epsilon_i\bigg|^{a'} + \bigg|e^{\frac{1}{2}Z_i^T\beta_0}\epsilon_i\bigg|^{-a'}.
\end{align*}
The last inequality follows from choosing a sufficiently small $l$ using the following bound:  For all real $l > 0$, there exists some constant $C_l$ such that for all $|u| \neq 0$
\begin{align*}
    \left|\log |u|\right| \leq C_l (|u|^l + |u|^{-l}).
\end{align*}
Taking expectation and supremum over $i$ and $n$, the first, third, and fourth terms are all finite under Assumptions \ref{(A2)} and \ref{(A3)}. Finally for sufficiently large $N$, by Assumption \ref{(A4)} we have 
\begin{equation*}
    \begin{split}
            &\sup_{n > N} \sup_{i} \mathbb{E}\left[\left|e^{\frac{1}{2}Z_i^T\beta_0}\epsilon_i - \frac{1}{\sqrt{n}}X_i^T\tilde u_n^{(0)}\right|^{-a'}\right]
            < \infty,
    \end{split}
\end{equation*}
Along with \eqref{eq7.9}, we obtain
\begin{equation}\label{eq7.11}
    M_{2,n} = \frac{2}{n}\sumi{\Zast _i\tilde \psi^{(0)}_{ni}I\left(\left|\frac{\tilde  r_{ni}^{(0)}}{\sqrt{n}}\right| >  \frac{1}{2}\right)} 
     = \mathcal{O}_p(n^{-\frac{s}{2}}).
\end{equation}
Combining the results of \eqref{eq7.8} and \eqref{eq7.11} proves \eqref{eq3.4}. If there exists $\tau > 0$ such that $\mathbb{E}\left[|\epsilon_1|^{-(1+\tau)}\right] < \infty$, then in \eqref{eq7.7} we can instead allow $s = 1$ and $b \leq 1 + \tau$ to get $M_{1,n} = \smallO_p(n^{-\frac{1}{2}})$ in \eqref{eq7.8} and in \eqref{eq7.9} let $1 < b <  t \leq 1 + \tau$ to obtain $M_{2,n} = \smallO_p(n^{-\frac{1}{2}})$, which show $\left(\frac{Z^{\ast T}\Zast }{n}+\frac{\Omega_n^\ast}{n}\right)^{-1}\frac{Z^{\ast T}R_n}{n^{1/2}} \to 0 $ in \eqref{eq7.4}. Finally, we deduce \eqref{eq3.5} using \eqref{eq7.5}.

\section{Proof of Theorem \ref{thm2}}\label{secB}
First, we write
\begin{align*}
    &\sqrt{n}(\tilde\alpha_n^{(1)} - \alpha_0)\\
    &= \sqrt{n}\left(\left(\frac{X^T D_n^{-2}(\tilde\beta_n^{(0)})X}{n} + \Lambda_n \frac{T(\tilde\alpha_n^{(0)})}{n}\right)^{-1}\frac{1}{n}X^T D_n^{-2}(\tilde\beta_n^{(0)})Y - \alpha_0 \right)\\
    &= \sqrt{n}\left(\left(\tilde\Sigma_{n-}^{(0)} + \Lambda_n \frac{T(\tilde\alpha_n^{(0)})}{n}\right)^{-1}\tilde\Sigma_{n-}^{(0)}\alpha_0 - \alpha_0 \right) \\
    &\quad+ \left(\tilde\Sigma_{n-}^{(0)} + \Lambda_n \frac{T(\tilde\alpha_n^{(0)})}{n}\right)^{-1}\frac{1}{\sqrt{n}}X^T D_n^{-2}(\tilde\beta_n^{(0)})D_n(\beta_0)\epsilon  \\
    &\equiv I_{1,n} + I_{2,n},
\end{align*}
where we denote $\tilde\Sigma_{n-}^{(0)} = \frac{X^T D_n^{-2}(\tilde\beta_n^{(0)})X}{n}$.  We will first prove $I_{1,n} \cip 0$ and then simplify the expression of $I_{2,n}$. 

First, if no entries are exactly 0, then $I_{1,n} \cip 0$ follows from the consistency of $\tilde\beta_n^{(0)}$, $\Lambda_n \frac{T(\tilde\alpha_n^{(0)})}{\sqrt{n}} \cip 0$ and the matrix inversion formula, which is similar to what we have proved in Theorem \ref{thm1}. Otherwise, suppose $p_0 < p$ are non-zero (The case where $p_0 = 0$ is similar to the result below without needing to consider block matrices), in which case we let 
\begin{align*}
    M &= \left(\frac{X^T D_n^{-2}(\tilde\beta_n^{(0)})X}{n} + \Lambda_n \frac{T(\tilde\alpha_n^{(0)})}{n}\right)\\
    &= \left(\tilde\Sigma_{n-}^{(0)} + \Lambda_n \frac{T(\tilde\alpha_n^{(0)})}{n}\right)\\
    &= \begin{pmatrix}
        (\tilde\Sigma_{n-}^{(0)})_{\star\star}+\Lambda_n\frac{T_\star(\tilde\alpha_n^{(0)})}{n} & (\tilde\Sigma_{n-}^{(0)})_{\star\circ}\\
        (\tilde\Sigma_{n-}^{(0) T})_{\star\circ} & (\tilde\Sigma_{n-}^{(0)})_{\circ\circ}+\Lambda_n\frac{T_\circ(\tilde\alpha_n^{(0)})}{n}
    \end{pmatrix}\\
     &=: \begin{pmatrix}
    A & B \\
    B^T  & D \\
    \end{pmatrix}.
\end{align*}
The Schur complement of $M$ is given by
\begin{align*}
    S &:= D - B^T A^{-1}B\\
      &=\left((\tilde\Sigma_{n-}^{(0)})_{\circ\circ}+\Lambda_n\frac{T_\circ(\tilde\alpha_n^{(0)})}{n}\right) - (\tilde\Sigma_{n-}^{(0)} )_{\star\circ}\left((\tilde\Sigma_{n-}^{(0)})_{\star\star}+\Lambda_n\frac{T_\star(\tilde\alpha_n^{(0)})}{n}\right)^{-1}(\tilde\Sigma_{n-}^{(0)})_{\star\circ} ,
\end{align*}
and the inverse of $M$ is (Lu and Shiou (2002) \cite{BlockMat})
\begin{equation}\label{eq7.12}
M^{-1}= \begin{pmatrix}
    M^{-1}_{\star\star} & M^{-1}_{\star\circ}\\
    (M^{-1}_{\star\circ})^T & M^{-1}_{\circ\circ}
    \end{pmatrix} =
\begin{pmatrix}
A^{-1} + A^{-1}BS^{-1}B^T A^{-1} & -A^{-1}BS^{-1} \\
-S^{-1}B^T A^{-1} & S^{-1}
\end{pmatrix}.
\end{equation}
Here, the matrices $(A,B,D,M,S) = (A_n,B_n,D_n,M_n,S_n)$ are written without the subscript $n$ for brevity. Next, we rewrite $\frac{1}{\sqrt{n}}I_{1,n}$ as follows:
\begin{equation}\label{eq7.13}
    \begin{split}
    \frac{1}{\sqrt{n}}I_{1,n}
    &= \frac{1}{\sqrt{n}}\begin{pmatrix}
        I_{1,n}'\\
        I_{1,n}''
    \end{pmatrix}\\
    &= M^{-1}\tilde\Sigma_{n-}^{(0)}\alpha_0 - \alpha_0\\
    &= \begin{pmatrix}
    M^{-1}_{\star\star} & M^{-1}_{\star\circ} \\
    (M^{-1}_{\star\circ})^T & M^{-1}_{\circ\circ}
    \end{pmatrix}
    \begin{pmatrix}
    (\tilde\Sigma_{n-}^{(0)})_{\star\star} & (\tilde\Sigma_{n-}^{(0)})_{\star\circ} \\
    (\tilde\Sigma_{n-}^{(0)})_{\star\circ}^T & (\tilde\Sigma_{n-}^{(0)})_{\circ\circ}
    \end{pmatrix}
    \begin{pmatrix}
    \alpha_{0\star} \\
    0
    \end{pmatrix} - \alpha_0 \\
    &= \begin{pmatrix}
    (M^{-1}_{\star\star} (\tilde\Sigma_{n-}^{(0)})_{\star\star} + M^{-1}_{\star\circ}(\tilde\Sigma_{n-}^{(0)})_{\star\circ}^T)\alpha_{0\star} \\
    \left((M^{-1}_{\star\circ})^T (\tilde\Sigma_{n-}^{(0)})_{\star\star} + M^{-1}_{\circ\circ}(\tilde\Sigma_{n-}^{(0)})_{\star\circ}^T\right)\alpha_{0\star}
    \end{pmatrix} - \begin{pmatrix}
    \alpha_{0\star}\\
    0
    \end{pmatrix}.
    \end{split}
\end{equation}
We will proceed to evaluate the limits of the rows
\begin{center}
    $\left(M^{-1}_{\star\star} (\tilde\Sigma_{n-}^{(0)})_{\star\star} + M^{-1}_{\star\circ}(\tilde\Sigma_{n-}^{(0)})_{\star\circ}^T\right)\alpha_{0\star} - \alpha_{0\star}$ 
\end{center}
and
\begin{center}
    $\left((M^{-1}_{\star\circ})^T (\tilde\Sigma_{n-}^{(0)})_{\star\star} + M^{-1}_{\circ\circ}(\tilde\Sigma_{n-}^{(0)})_{\star\circ}^T\right)\alpha_{0\star}$
\end{center}
separately. For the upper part $I_{1,n}'$, note that $(\tilde\Sigma_{n-}^{(0)})_{\star\circ} = B$, so by \eqref{eq7.12} we get 
\begin{equation}\label{eq7.14}
    \begin{split}
        &M^{-1}_{\star\star}(\tilde\Sigma_{n-}^{(0)})_{\star\star}+M^{-1}_{\star\circ}(\tilde\Sigma_{n-}^{(0)})_{\star\circ}^T\\
        =& 
    (A^{-1}+A^{-1}BS^{-1}B^T A^{-1})(\tilde\Sigma_{n-}^{(0)})_{\star\star} - A^{-1}BS^{-1}B^T \\
    =& A^{-1}(\tilde\Sigma_{n-}^{(0)})_{\star\star} + A^{-1}BS^{-1}B^T (A^{-1}(\tilde\Sigma_{n-}^{(0)})_{\star\star}-I_{p_0}).
    \end{split}
\end{equation}
Moreover by \eqref{eq7.1}, since $\Lambda_n\frac{T_\star(\tilde\alpha_n^{(0)})}{n} = \mathcal{O}_p(1/n)$, 
\begin{equation*}
    A^{-1}(\tilde\Sigma_{n-}^{(0)})_{\star\star} = \left((\tilde\Sigma_{n-}^{(0)})_{\star\star}+\Lambda_n\frac{T_\star(\tilde\alpha_n^{(0)})}{n}\right)^{-1}(\tilde\Sigma_{n-}^{(0)})_{\star\star} = I_{p_0} + \mathcal{O}_p\left(\frac{1}{n}\right).
\end{equation*}
Finally by Assumption \ref{(A1)}, $\Vert B \Vert$ is bounded for large enough $n$, and since $\Vert A^{-1} \Vert \leq \left\Vert \left(\tilde\Sigma_{n-}^{(0)}\right)^{-1}\right\Vert$, $\Vert A^{-1} \Vert$ is bounded for all large enough $n$. Also, $\Vert M^{-1} \Vert \leq K_X$ so $\Vert S^{-1} \Vert$ is also bounded for all sufficiently large $n$, which jointly imply
\[A^{-1}BS^{-1}B^T  = \mathcal{O}_p(1).\]
Therefore,
\begin{align*}
    M^{-1}_{\star\star}(\tilde\Sigma_{n-}^{(0)})_{\star\star}+M^{-1}_{\star\circ}(\tilde\Sigma_{n-}^{(0)})_{\star\circ}^T = I_{p_0} + \mathcal{O}_p\left(\frac{1}{n}\right),
\end{align*}
and thus
\begin{equation}\label{eq7.16}
    \sqrt{n}\left((M^{-1}_{\star\star}(\tilde\Sigma_{n-}^{(0)})_{\star\star}+M^{-1}_{\star\circ}(\tilde\Sigma_{n-}^{(0)})_{\star\circ}^T)\alpha_{0\star} - \alpha_{0\star} \right) \cip 0.
\end{equation}
For the lower part $I_{1,n}''$, we get 
\begin{equation}\label{eq7.17}
    \begin{split}
        (M^{-1}_{\star\circ})^T (\tilde\Sigma_{n-}^{(0)})_{\star\star} + M^{-1}_{\circ\circ}(\tilde\Sigma_{n-}^{(0)})_{\star\circ}^T &= -S^{-1}B^T A^{-1}(\tilde\Sigma_{n-}^{(0)})_{\star\star} + S^{-1}(\tilde\Sigma_{n-}^{(0)})_{\star\circ}^T \\
    & = S^{-1}B^T (I_{p_0}-A^{-1}(\tilde\Sigma_{n-}^{(0)})_{\star\star})\\
    & = S^{-1}B^T \left(I_{p_0}-I_{p_0}+\mathcal{O}_p\left(\frac{1}{n}\right)\right) \\
    & = \mathcal{O}_p\left(\frac{1}{n}\right).
    \end{split}
\end{equation}
Therefore,
\begin{equation}\label{eq7.18}
    \sqrt{n}\left((M^{-1}_{\star\circ})^T (\tilde\Sigma_{n-}^{(0)})_{\star\star} + M^{-1}_{\circ\circ}(\tilde\Sigma_{n-}^{(0)})_{\star\circ}^T\right)\alpha_{0\star} \cip 0.
\end{equation}
Substituting \eqref{eq7.16} and \eqref{eq7.18} into \eqref{eq7.13} gives 
\begin{center}
    $I_{1,n} = \smallO_p(1)$
\end{center}
as required. 

Moving onto $I_{2,n}$, the consistency of $\tilde\beta_n^{(0)}$ and $\sup_n \Vert Z \Vert < \infty$ (Assumption \ref{(A2)}) imply

    \[\max_{1\leq i \leq n} |e^{-Z_i^T\beta_0} - e^{-Z_i^T \tilde\beta_n^{(0)}}| \cip 0,\]

\[
    \tilde\Sigma_{n-}^{(0)} - \frac{X^TD^{-2}_n(\beta_0)X}{n} \cip 0
\]
and 
\[
    X^T D_n^{-2}(\tilde\beta_n^{(0)})D_n(\beta_0)\epsilon - X^T D_n^{-1}(\beta_0)\epsilon \cip 0.
\]
Writing $\Pi^X_n \coloneqq \frac{1}{\sqrt{n}}X^T D_n^{-1}(\beta_0)\epsilon$ then gives
\begin{equation}\label{eq7.19}
    \begin{split}
        I_{2,n} &= \left(\tilde\Sigma_{n-}^{(0)} + \Lambda_n \frac{T(\tilde\alpha_n^{(0)})}{n}\right)^{-1}\frac{1}{\sqrt{n}}X^T D_n^{-2}(\tilde\beta_n^{(0)})D_n(\beta_0)\epsilon\\
    &= \left(\frac{X^TD^{-2}_n(\beta_0)X}{n} + \Lambda_n \frac{T(\tilde\alpha_n^{(0)})}{n}\right)^{-1}\Pi^X_n + \smallO_p(1).
    \end{split}
\end{equation}
We can use the Lindeberg CLT analogously as we did for \eqref{eq7.3} to obtain
\[\Pi^X_n \cil N(0, \Sigma_-^X).\]
Further, by noting $\left(\frac{X^TD^{-2}_n(\beta_0)X}{n}-\Sigma_-^X\right) = \smallO(1)$ and $\Lambda_n \frac{T(\tilde\alpha_n^{(0)})}{n} - \begin{pmatrix}
        0 & 0 \\
        0 & \Lambda_{0\circ}T_\circ(\tilde u_n^{(0)})
    \end{pmatrix} = \smallO_p(1)$, we may appeal to \eqref{eq7.1} and $\Pi_n^X = \mathcal{O}_p(1)$ to conclude
\begin{align*}
    &\left(\frac{X^TD^{-2}_n(\beta_0)X}{n} + \Lambda_n \frac{T(\tilde\alpha_n^{(0)})}{n}\right)^{-1}\Pi^X_n \\
    =& \Bigg(\Sigma_-^X  + \begin{pmatrix}
        0 & 0 \\
        0 & \Lambda_{0\circ}T_\circ(\tilde u_n^{(0)})
    \end{pmatrix} +\left(\frac{X^TD^{-2}_n(\beta_0)X}{n}-\Sigma_-^X\right)\\
    &\quad+ \left(\Lambda_n \frac{T(\tilde\alpha_n^{(0)})}{n} - \begin{pmatrix}
        0 & 0 \\
        0 & \Lambda_{0\circ}T_\circ(\tilde u_n^{(0)})
    \end{pmatrix}\right)\Bigg)^{-1}\Pi^X_n\\
    = & \left(\Sigma_-^X + \begin{pmatrix}
        0 & 0 \\
        0 & \Lambda_{0\circ}T_\circ(\tilde u_n^{(0)})
    \end{pmatrix}\right)^{-1}\Pi^X_n + \smallO_p(1).
\end{align*}
Returning to \eqref{eq7.19}, we obtain
\begin{equation}\label{eq7.20}
    \begin{split}
        I_{2,n} &= \left(\frac{X^TD^{-2}_n(\beta_0)X}{n} + \Lambda_n \frac{T(\tilde\alpha_n^{(0)})}{n}\right)^{-1}\Pi^X_n + \smallO_p(1)\\
    &= \left(\Sigma_-^X + \begin{pmatrix}
        0 & 0 \\
        0 & \Lambda_{0\circ}T_\circ(\tilde u_n^{(0)})
    \end{pmatrix}\right)^{-1}\Pi^X_n + \smallO_p(1),
    \end{split}
\end{equation}
which proves \eqref{eq3.8}. 

We now move on to prove the corresponding results for $\tilde\beta^{\ast(1)}_n$ by first writing
\begin{align*}
    &\tilde\beta_n^{\ast(1)}-\basttrue\\
    &= \left(\frac{Z^{\ast T}\Zast }{n} + \gamastn\frac{S(\bastinit)}{n}\right)^{-1}\frac{1}{n}Z^{\ast T}L_n(\tilde\alpha_n^{(1)}) - \basttrue\\
    &= \left[\left(\frac{Z^{\ast T}\Zast }{n} + \gamastn\frac{S(\bastinit)}{n}\right)^{-1}\frac{Z^{\ast T}\Zast }{n}\basttrue - \basttrue\right]\\
    &+ \left(\frac{Z^{\ast T}\Zast }{n} + \gamastn\frac{S(\bastinit)}{n}\right)^{-1}\frac{1}{n}Z^{\ast T}E_n + \left(\frac{Z^{\ast T}\Zast }{n} + \gamastn\frac{S(\bastinit)}{n}\right)^{-1}\frac{1}{n}Z^{\ast T}R_n^{(1)}\\
     &\equiv J_{1,n} + J_{2,n} +J_{3,n},
\end{align*}
where 
\begin{center}
    $R_n^{(1)} := L(\tilde\alpha^{(1)}_n) - L_n(\alpha_0)$.
\end{center}

The term $J_{1,n}$ can be shown to be $\smallO_p(\frac{1}{\sqrt{n}})$ in the exact same way we did for $I_{1,n}$, swapping out $\tilde\Sigma_{n-}^{(0)}$ for $\frac{Z^{\ast T}\Zast }{n}$; $\Lambda_n$ for $\Gamma_n^\ast$; $\tilde\alpha_n^{(0)}$ for $\bastinit$ and $\alpha_0$ for $\basttrue$. Then under Assumption \ref{(A3)}, we denote $\Pi_n^Z \coloneqq \frac{1}{\sqrt{n}}Z^{\ast T}E_n$ and use the Lindeberg CLT to show \[\Pi_n^Z \cil N\left(0, \Sigma^{\Zast}\var[\log(\epsilon_1^2)]\right),\]
which also implies $J_{2,n} = O_p(\frac{1}{\sqrt{n}})$.

Finally, under Lemma \ref{lem2} (stated below), we can follow the proof of Theorem \ref{thm1} to show that $J_{3,n} = O_p(n^{-\frac{c}{2}})$ for all $c \in (0,1)$ by bounding the term $\frac{1}{n}Z^{\ast T}R_n^{(1)}$, which gives \eqref{eq3.9}. Further, with the stronger condition of $\mathbb{E}\left[|\epsilon_1|^{-(1+\tau)}\right] < \infty$ for some $\tau > 0$, we can proceed as in \eqref{eq7.20} to get 
\begin{align*}
    \sqrt{n}J_{2,n} &=  \left(\frac{Z^{\ast T}\Zast }{n} + \gamastn\frac{S(\bastinit)}{n}\right)^{-1} \Pi_n^Z\\
    &= \left(\Sigma^{\Zast} + \begin{pmatrix}
    0 & 0\\
    0 & \Gamma_{0\circ}^\ast S_\circ(v_n^{(0)})
\end{pmatrix}\right)^{-1} \Pi_n^Z + \smallO_p(1),
\end{align*} 
which gives \eqref{eq3.10}. The statement for general $k \geq 2$ follows inductively. 

We state a similar lemma as Lemma \ref{lem1} to prove corresponding results for $\tilde\beta^{\ast(k)}_n$. 

\begin{lem}\label{lem2}
    For each $k \geq 1$, there exists some $N$ such that for all $s > 1$, \[ \sup_{n > N} \mathbb{E}\left[ \left \Vert \sqrt{n}(\tilde\alpha_n^{(k)} - \alpha_0) \right \Vert^{s}\right] < \infty.\]
\end{lem}
\proof
    We only consider $k = 1$ as the general case follows the same proof. We write $\sqrt{n}(\tilde\alpha_n^{(1)} - \alpha_0) = I_{1,n} + I_{2,n}$ as before, and we will bound the moments of each of the two terms.
    
    Using the same notation as in \eqref{eq7.13}, we may write $I_{1,n}$ as follows and bound the upper and lower parts separately:
    \begin{equation*}
        \begin{split}
            I_{1,n} &= \sqrt{n}\left\{\begin{pmatrix}
    \left(M^{-1}_{\star\star} (\tilde\Sigma_{n-}^{(0)})_{\star\star} + M^{-1}_{\star\circ}(\tilde\Sigma_{n-}^{(0)})_{\star\circ}^T\right)\alpha_{0\star} \\
    \left((M^{-1}_{\star\circ})^T (\tilde\Sigma_{n-}^{(0)})_{\star\star} + M^{-1}_{\circ\circ}(\tilde\Sigma_{n-}^{(0)})_{\star\circ}^T\right)\alpha_{0\star}
    \end{pmatrix} - \begin{pmatrix}
    \alpha_{0\star}\\
    0
    \end{pmatrix}\right\} \\
    &=  \begin{pmatrix}
    \sqrt{n} \left(M^{-1}_{\star\star} (\tilde\Sigma_{n-}^{(0)})_{\star\star} + M^{-1}_{\star\circ}(\tilde\Sigma_{n-}^{(0)})_{\star\circ}^T - I_{p_0}\right)\alpha_{0\star} \\
    \sqrt{n} \left((M^{-1}_{\star\circ})^T (\tilde\Sigma_{n-}^{(0)})_{\star\star} + M^{-1}_{\circ\circ}(\tilde\Sigma_{n-}^{(0)})_{\star\circ}^T\right)\alpha_{0\star}
    \end{pmatrix}
    =:\begin{pmatrix}
        I_{1,n}' \\ I_{1,n}''
    \end{pmatrix}.
        \end{split}
    \end{equation*}
First, we look at $I_{1,n}'$. As we previously expanded in \eqref{eq7.14}, for $s > 1$  
\begin{align*}
    &\mathbb{E}\left[\left\Vert \sqrt{n}\left(M^{-1}_{\star\star} (\tilde\Sigma_{n-}^{(0)})_{\star\star} + M^{-1}_{\star\circ}(\tilde\Sigma_{n-}^{(0)})_{\star\circ}^T - I_{p_0}\right)\right\Vert^s\right]\\
    &= \mathbb{E}\left[\left\Vert \sqrt{n}\left(A^{-1}(\tilde\Sigma_{n-}^{(0)})_{\star\star} - I_{p_0} + A^{-1}BS^{-1}B^T \left(A^{-1}(\tilde\Sigma_{n-}^{(0)}\right)_{\star\star}-I_{p_0})\right)\right\Vert^s\right]\\
    \lesssim & \mathbb{E}\left[\left\Vert \sqrt{n}\left(A^{-1}(\tilde\Sigma_{n-}^{(0)})_{\star\star} - I_{p_0}\right)\right\Vert^s\right] 
    \nn\\
    &{}\qquad + \mathbb{E}\left[\left\Vert \sqrt{n}\left(A^{-1}BS^{-1}B^T \left(A^{-1}(\tilde\Sigma_{n-}^{(0)})_{\star\star}-I_{p_0}\right)\right)\right\Vert^s\right].
\end{align*}
By the matrix inversion formula \eqref{eq7.1}, we also have 
\begin{align*}
    &\sqrt{n} \left(A^{-1}(\tilde\Sigma_{n-}^{(0)})_{\star\star} - I_{p_0}\right)\\
    &= \sqrt{n}\left\{\left((\tilde\Sigma_{n-}^{(0)})_{\star\star}+\Lambda_n\frac{T_\star(\tilde\alpha_n^{(0)})}{n}\right)^{-1}(\tilde\Sigma_{n-}^{(0)})_{\star\star} - I_{p_0}\right\}\\    
    &= -\left((\tilde\Sigma_{n-}^{(0)})_{\star\star}\right)^{-1}\Lambda_n\frac{T_\star(\tilde\alpha_n^{(0)})}{\sqrt{n}}\left((\tilde\Sigma_{n-}^{(0)})_{\star\star}+\Lambda_n\frac{T_\star(\tilde\alpha_n^{(0)})}{n}\right)^{-1}(\tilde\Sigma_{n-}^{(0)})_{\star\star}. 
\end{align*}
Using the eigenvalue bound of $\left(\frac{X^TD_n^{-2}(\beta)X}{n}\right)$ in Assumption \ref{(A1)}, we can also bound the moments of $(\tilde\Sigma_{n-}^{(0)})_{\star\star}$ and $\left((\tilde\Sigma_{n-}^{(0)})_{\star\star}\right)^{-1}$ for sufficiently large $n$. Then, rewriting \[\frac{T_\star(\tilde\alpha_n^{(0)})}{\sqrt{n}} = \left\{\diag\left(\frac{1}{n^{1/4}(\tilde\alpha^{(0)}_{n1} - \alpha_{01}) + n^{1/4}\alpha_{01}}, ... , \frac{1}{n^{1/4}(\tilde\alpha^{(0)}_{np_0} - \alpha_{0p_0}) + n^{1/4}\alpha_{0p_0}}\right)\right\}^2,\]
we will consider conditioning on the event 
\[H_n = \{
\Vert n^{1/4}(\tilde\alpha^{(0)}_{n\star} - \alpha_{0\star})\Vert \leq h_0\},\]
for a fixed constant $h_0 > 0$, so that
\begin{align*}
&\sup_n \mathbb{E}\left[\left\Vert \sqrt{n} \left(A^{-1}(\tilde\Sigma_{n-}^{(0)})_{\star\star} - I_{p_0}\right) \right\Vert^s\right] \\
&= \sup_n \mathbb{E}\left[\left\Vert \sqrt{n} \left(A^{-1}(\tilde\Sigma_{n-}^{(0)})_{\star\star} - I_{p_0}\right) \right\Vert^s ; H_n\right] 
\nn\\
&{}\qquad + \sup_n \mathbb{E}\left[\left\Vert \sqrt{n} \left(A^{-1}(\tilde\Sigma_{n-}^{(0)})_{\star\star} - I_{p_0}\right) \right\Vert^s ; H^c_n\right].
\end{align*}
Since $n^{1/4}\alpha_{0i} \to \pm\infty$ for $1 \leq i \leq p_0$,$\frac{T_\star(\tilde\alpha_n^{(0)})}{\sqrt{n}}$ is bounded on $H_n$  for sufficiently large $n$, which in turn ensures that $\sup_{n > N} \mathbb{E}\left[\left\Vert \sqrt{n} \left(A^{-1}(\tilde\Sigma_{n-}^{(0)})_{\star\star} - I_{p_0}\right) \right\Vert^s ; H_n\right] < \infty$ for some $N \geq 1$. As $\left\Vert\left(A^{-1}(\tilde\Sigma_{n-}^{(0)})_{\star\star} - I_{p_0}\right)\right\Vert$ is bounded,
\begin{align*}
\mathbb{E}\left[\left\Vert \sqrt{n} \left(A^{-1}(\tilde\Sigma_{n-}^{(0)})_{\star\star} - I_{p_0}\right) \right\Vert^s ; H^c_n\right]
&\lesssim n^{s/2}\mathbb{P}\left(\Vert n^{1/4}(\tilde\alpha^{(0)}_{n\star} - \alpha_{0\star})\Vert > h_0\right) \\
&\leq n^{s/2} \frac{n^{-m/4}\mathbb{E}\left[\Vert n^{1/2}(\tilde\alpha^{(0)}_{n\star} - \alpha_{0\star})\Vert^m\right]}{h_0^m},
\end{align*}
where the latter inequality follows from Markov's inequality. By choosing $m > 2s$, along with Lemma \ref{lem1}, we can conclude 
\[\sup_n \mathbb{E}\left[\left\Vert \sqrt{n} \left(A^{-1}(\tilde\Sigma_{n-}^{(0)})_{\star\star} - I_{p_0}\right) \right\Vert^s\right] < \infty.\]
Finally, by uniformly bounding the norms of $A, B, B^{T}$ and $S^{-1}$ as we did in the proof of Theorem \ref{thm2}, we also get 
\[\sup_{n > N} \mathbb{E}\left[\left\Vert \sqrt{n}\left(A^{-1}BS^{-1}B^T \left(A^{-1}(\tilde\Sigma_{n-}^{(0)})_{\star\star}-I_{p_0}\right)\right)\right\Vert^s\right] < \infty.\]
Next, we turn to $I_{1,n}''$. Once again, it suffices to show that \[\sup_{n>N} \mathbb{E}\left[\left\Vert  \sqrt{n} \left((M^{-1}_{\star\circ})^T (\tilde\Sigma_{n-}^{(0)})_{\star\star} + M^{-1}_{\circ\circ}(\tilde\Sigma_{n-}^{(0)})_{\star\circ}^T\right)\right\Vert^s\right] < \infty.\]
By the second line of \eqref{eq7.17}, we can rewrite the expression as
\[\sup_{n>N} \mathbb{E}\left[\left\Vert  \sqrt{n} \left(S^{-1}B^T \left(I_{p_0}-A^{-1}(\tilde\Sigma_{n-}^{(0)})_{\star\star}\right)\right)\right\Vert^s\right] < \infty,\]
which can be handled analogously to $I_{1,n}'$.
We conclude that $\sup_n\mathbb{E}[\| I_{1,n}\|^s]<\infty$.

We next need to show that $\sup_{n>N} \mathbb{E}\left[ \Vert I_{2,n}\Vert^s\right] < \infty.$ Recall that \[I_{2,n} = \left(\tilde\Sigma_{n-}^{(0)} + \Lambda_n \frac{T(\tilde\alpha_n^{(0)})}{n}\right)^{-1}\frac{1}{\sqrt{n}}X^T D_n^{-2}(\tilde\beta_n^{(0)})D_n(\beta_0)\epsilon,\]
and by Assumption \ref{(A1)}\[\left\Vert \left(\tilde\Sigma_{n-}^{(0)} + \Lambda_n \frac{T(\tilde\alpha_n^{(0)})}{n}\right)^{-1} \right\Vert < \left\Vert \left(\tilde\Sigma_{n-}^{(0)}\right)^{-1}\right\Vert \leq K_X.\] Therefore, it suffices to show 
\begin{align*}
    & \sup_{n>N} \mathbb{E}\left[ \left\Vert \frac{1}{\sqrt{n}}X^T D_n^{-2}(\tilde\beta_n^{(0)})D_n(\beta_0)\epsilon \right\Vert^s\right] 
    \nn\\
    &\leq  \sup_{n>N} \mathbb{E}\left[ \sup_\beta\left\Vert \frac{1}{\sqrt{n}}X^T D_n^{-2}(\beta)D_n(\beta_0)\epsilon \right\Vert^s\right] < \infty.
\end{align*}
Let $H_n(\beta) = \frac{1}{\sqrt{n}}X^T D_n^{-2}(\beta)D_n(\beta_0)\epsilon$.
As $\Theta_\beta$ is a bounded convex domain in $\mathbb{R}^q$, 
the Sobolev inequality ensures that for any $s > q\vee 2$, we have
\[
\mathbb{E}\left[ \sup_\beta \Vert H_n(\beta)\Vert^s\right] \leq C_{\Theta_\beta,q} \left(\sup_\beta\mathbb{E}\left[\Vert H_n(\beta)\Vert^s\right] + \sup_\beta\mathbb{E}\left[\Vert H_n'(\beta)\Vert^s\right]\right),\]
where $C_{\Theta_\beta,q}$ is a constant independent of $n$, and $H_n'(\beta)$ is the partial derivative of $H_n$ with respect to $\beta$. We can then write $H_n(\beta) = \frac{1}{\sqrt{n}}\sumi \chi_{ni}(\beta)\epsilon_i$ and $H_n(\beta) = \frac{1}{\sqrt{n}}\sumi \chi_{ni}'(\beta)\epsilon_i$, where both $\chi_i(\beta)$ and $\chi_i'(\beta)$ are non-random and 
satisfy that $\sup_{n\ge n_0}\max_{i\le n}\|\chi_{ni}(\beta)\|_{\infty} \vee \|\chi'_{ni}(\beta)\|_{\infty} <\infty$ for $n_0$ large enough because $\Vert X_{ni}\Vert$ and $\Vert Z_{ni}\Vert$ are both bounded by Assumption \ref{(A2)}. Finally, the Burkholder inequality gives
\begin{align*}
\sup_\beta \mathbb{E}\left[\left\Vert \frac{1}{\sqrt{n}}\sumi \chi_{ni}(\beta)\epsilon_i \right\Vert^s\right] &\lesssim \sup_\beta \frac{1}{n} \sumi \mathbb{E}\left[\Vert \chi_i(\beta)\epsilon_i\Vert^s\right]\\
&\lesssim \frac{1}{n}\sumi \mathbb{E}\left[\Vert \epsilon_i\Vert^s\right] < \infty,
\end{align*}
where the final inequality follows from $\mathbb{E}[|\epsilon_1|^s] < \infty$ by Assumption \ref{(A3)}. The result corresponding to $\chi_{ni}'(\beta)$ follows analogously.
\qed

\section{Proof of Theorem \ref{thm3}}\label{secC}
We only show the result for $k=0$, the general case can be obtained similarly. A similar proof is also shown in Dai et al. (2018) \cite{BAR2018}. Denoting $\tilde\Sigma^{(0)}_{n-} := n^{-1}X^T D_n^{-2}(\tilde\beta_n^{(0)})X$, we write  
\begin{align*}
    \tilde\alpha_n^{(1)} &= \left(\tilde\Sigma^{(0)}_{n-}+\frac{\Lambda_n}{n} T(\tilde\alpha_n^{(0)})\right)^{-1}\frac{1}{n}X^TD^{-2}_n(\tilde\beta_n^{(0)})Y\\
    &= \left(\tilde\Sigma^{(0)}_{n-}+\frac{\Lambda_n}{n} T(\tilde\alpha_n^{(0)})\right)^{-1}\tilde\Sigma^{(0)}_{n-}\alpha_0 
    \nn\\
    &{}\qquad + \left(\tilde\Sigma^{(0)}_{n-}+\frac{\Lambda_n}{n} T(\tilde\alpha_n^{(0)})\right)^{-1}\frac{1}{n}X^T D_n^{-2}(\tilde\beta_n^{(0)})D_n(\beta_0)\epsilon.
\end{align*}
Multiplying both sides by $ \left(\tilde\Sigma^{(0)}_{n-}\right)^{-1}\left(\tilde\Sigma^{(0)}_{n-}+\frac{\Lambda_n}{n} T(\tilde\alpha_n^{(0)})\right)$ yields
\begin{align*}              
\left(\tilde\Sigma^{(0)}_{n-}\right)^{-1}\left(\tilde\Sigma^{(0)}_{n-}+\frac{\Lambda_n}{n} T(\tilde\alpha_n^{(0)})\right)\tilde\alpha_n^{(1)} = \alpha_0 + \left(\tilde\Sigma^{(0)}_{n-}\right)^{-1}\frac{X^T D_n^{-2}(\tilde\beta_n^{(0)})D_n(\beta_0)\epsilon}{n},
\end{align*}
which can be rearranged to
\begin{align*}
    \tilde\alpha_n^{(1)} - \alpha_0 + \left(\tilde\Sigma^{(0)}_{n-}\right)^{-1}\frac{\Lambda_n}{n} T(\tilde\alpha_n^{(0)})\tilde\alpha_n^{(1)} = \left(\tilde\Sigma^{(0)}_{n-}\right)^{-1}\frac{X^T D_n^{-2}(\tilde\beta_n^{(0)})D_n(\beta_0)\epsilon}{n},
\end{align*}
namely
\begin{equation}\label{eq7.22}
    \begin{pmatrix}
    \tilde\alpha_{n\star}^{(1)} - \alpha_{0\star} \\
    \tilde\alpha_{n\circ}^{(1)}
    \end{pmatrix}
    + \left(\tilde\Sigma^{(0)}_{n-}\right)^{-1}\frac{\Lambda_n}{n} T(\tilde\alpha_n^{(0)})
    \begin{pmatrix}
    \tilde\alpha_{n\star}^{(1)}\\
    \tilde\alpha_{n\circ}^{(1)}
    \end{pmatrix}
    =\left(\tilde\Sigma^{(0)}_{n-}\right)^{-1}\frac{X^T D_n^{-2}(\tilde\beta_n^{(0)})D_n(\beta_0)\epsilon}{n}.
\end{equation}
As $\left(\tilde\Sigma^{(0)}_{n-}\right)^{-1} = \mathcal{O}_p(1)$ and $n^{-1}X^TD^{-2}_n(\tilde\beta_n^{(0)})D_n(\beta_0)\epsilon = \mathcal{O}_p\left(\frac{1}{\sqrt{n}}\right)$, the right-hand side of \eqref{eq7.22} is $\mathcal{O}_p\left(\frac{1}{\sqrt{n}}\right)$. For ease of reading, we now let $S_{n}:=\left(\tilde\Sigma^{(0)}_{n-}\right)^{-1}$ for the remainder of the proof. Then, focusing on the lower $(p - p_0)$ components of \eqref{eq7.22} gives
\begin{equation}\label{eq7.23}
    \left\Vert \tilde\alpha^{(1)}_{n\circ} + \frac{1}{n}(S_{n})_{\star\circ}^T\Lambda_{n\star}T_\star(\tilde\alpha_n^{(0)})\tilde\alpha^{(1)}_{n\star} + \frac{1}{n}(S_{n})^{-1}_{\circ\circ}\Lambda_{n\circ}T_\circ(\tilde\alpha_n^{(0)})\tilde\alpha_{n\circ}^{(1)}\right\Vert = \mathcal{O}_p\left(\frac{1}{\sqrt{n}}\right).
\end{equation}
We then proceed similarly as in Lemma 1 of Dai et al.(2018) \cite{BAR2018} by first considering the order of the second term on the left-hand side of \eqref{eq7.23}: 
\begin{align*}
   \left \Vert \frac{1}{n}(S_{n})_{\star\circ}^T\Lambda_{n\star}T_\star(\tilde\alpha_n^{(0)})\tilde\alpha^{(1)}_{n\star} \right\Vert &\leq \frac{1}{n} \left \Vert (S_{n})_{\star\circ}^T \right \Vert  \Vert \Lambda_{n\star} \Vert \Vert  T_\star(\tilde\alpha_n^{(0)}) \Vert \Vert \tilde\alpha^{(1)}_{n\star} \Vert\\
    &= \frac{1}{\sqrt{n}} \left \Vert (S_{n})_{\star\circ}^T \right \Vert \left\Vert \frac{\Lambda_{n\star}}{\sqrt{n}} \right\Vert \Vert  T_\star(\tilde\alpha_n^{(0)}) \Vert \Vert \tilde\alpha^{(1)}_{n\star} \Vert\\
    &= \smallO_p\left(\frac{1}{\sqrt{n}}\right),
\end{align*}
where $\left\Vert (S_{n})_{\star\circ}^T \right\Vert = \mathcal{O}_p(1)$ follows from Assumption \ref{(A1)} as we also note $\Vert  T_\star(\tilde\alpha_n^{(0)}) \Vert = \smallO_p(1)$ and $\left\Vert \frac{\Lambda_{n\star}}{\sqrt{n}} \right\Vert = \smallO(1)$. Hence,
\begin{equation}\label{eq7.24}
    \left \Vert \tilde\alpha^{(1)}_{n\circ} + \frac{1}{n}\left(\tilde\Sigma^{(0)}_{n-}\right)^{-1}_{\circ\circ}\Lambda_{n\circ}T_\circ(\tilde\alpha_n^{(0)})\tilde\alpha_{n\circ}^{(1)} \right \Vert = \mathcal{O}_p\left(\frac{1}{\sqrt{n}}\right).
\end{equation}
Using the triangle inequality, we estimate the left-hand side of \eqref{eq7.24} as follows:
\[\left\Vert \tilde\alpha^{(1)}_{n\circ} + \frac{1}{n}(S_{n})_{\circ\circ}\Lambda_{n\circ}T_\circ(\tilde\alpha_n^{(0)})\tilde\alpha_{n\circ}^{(1)} \right\Vert \geq \left\Vert \frac{1}{n}(S_{n})_{\circ\circ}\Lambda_{n\circ}T_\circ(\tilde\alpha_n^{(0)})\tilde\alpha_{n\circ}^{(1)} \right\Vert - \Vert \tilde\alpha^{(1)}_{n\circ} \Vert.\]

Below, we will bound the norm of $n^{-1}(S_{n})_{\circ\circ}\Lambda_{n\circ}T_\circ(\tilde\alpha_n^{(0)})\tilde\alpha_{n\circ}^{(1)}$ from below and the norm of $\tilde\alpha^{(1)}_{n\circ}$ from above in \eqref{eq7.25} and \eqref{eq7.26}, respectively.

As $(S_{n})^{-1}_{\circ\circ}$ is symmetric positive definite, it is diagonalizable by orthonormal eigenvectors. 
With probability tending to one, $(S_{n})_{\circ\circ}$ has a spectral decomposition $(S_{n})_{\circ\circ} = \sum_{i=1}^{p-p_0}\tau_ie_ie_i^T$ for eigenvalues $\tau_i \in [\frac{1}{K_X},{K_X}]$ and orthonormal eigenvectors $e_i$ by Assumption \ref{(A1)}. Using the spectral decomposition, we obtain
\begin{align*}
    &\left \Vert (S_{n})_{\circ\circ}\Lambda_{n\circ}T_\circ(\tilde\alpha_n^{(0)})\tilde\alpha_{n\circ}^{(1)} \right \Vert\\
    =& \left \Vert \sum_{i=1}^{p-p_0}\tau_ie_ie_i^T \Lambda_{n\circ}T_\circ(\tilde\alpha_n^{(0)})\tilde\alpha_{n\circ}^{(1)} \right \Vert \\
    =& \left(\sum_{i,j=1}^{p-p_0}\tau_i\tau_j \tilde\alpha_{n\circ}^{(1) T}\Lambda_{n\circ}T_\circ(\tilde\alpha_n^{(0)})e_ie_i^T e_je_j^T \Lambda_{n\circ}T_\circ(\tilde\alpha_n^{(0)})\tilde\alpha_{n\circ}^{(1)}\right)^{1/2}\\
    =& \left(\sum_{i=1}^{p-p_0}\tau_i^2\tilde\alpha_{n\circ}^{(1)T} \Lambda_{n\circ}T_\circ(\tilde\alpha_n^{(0)})e_ie_i^T \Lambda_{n\circ}T_\circ(\tilde\alpha_n^{(0)})\tilde\alpha_{n\circ}^{(1)}\right)^{1/2}\\ 
    \geq& \frac{1}{K_X} \Vert \Lambda_{n\circ}T_\circ(\tilde\alpha_n^{(0)})\tilde\alpha_{n\circ}^{(1)} \Vert.
\end{align*} 
Now denoting \[d_{n\circ}^{(1)} := \left(\frac{\tilde\alpha^{(1)}_{n,p_0+1}}{\tilde\alpha^{(0)}_{n,p_0+1}},...,\frac{\tilde\alpha^{(1)}_{n,p}}{\tilde\alpha^{(0)}_{n,p}}\right),\]
with probability tending to 1, we have
\begin{equation}\label{eq7.25}
\begin{split}
    &\left \Vert \frac{1}{n} (S_{n})_{\circ\circ} \Lambda_{n\circ}T_\circ(\tilde\alpha_n^{(0)})\tilde\alpha_{n\circ}^{(1)} \right \Vert \\
    \geq& \frac{1}{K_X} \left \Vert \frac{1}{n} \Lambda_{n\circ} T_\circ(\tilde\alpha_n^{(0)})\tilde\alpha_{n\circ}^{(1)} \right \Vert\\
    =& \frac{\min_{p_0<j\leq p}\lambda_{nj}}{K_X} \left \Vert \frac{1}{n} T^{1/2}_\circ(\tilde\alpha_n^{(0)})d_{n\circ}^{(1)} \right \Vert\\
    \geq& \frac{1}{n}\frac{\min_{p_0<j\leq p}\lambda_{nj}}{K_X}\frac{1}{\max_{p_0<j\leq p}|\tilde\alpha_{nj}^{(0)}|} \Vert d^{(1)}_{n\circ} \Vert.
\end{split}
\end{equation}
Moving on to $\Vert \tilde\alpha^{(1)}_{n\circ} \Vert$, we have 
\begin{equation}\label{eq7.26}
    \begin{split}
    \Vert \tilde\alpha^{(1)}_{n\circ} \Vert &= \Vert T^{-1/2}_\circ(\tilde\alpha_n^{(0)})d_{n\circ}^{(1)} \Vert
    \leq \max_{p_0<j\leq p}|\tilde\alpha_{nj}^{(0)}| \Vert d^{(1)}_{n\circ} \Vert. \\      
    \end{split}
\end{equation}
Combining \eqref{eq7.25} and \eqref{eq7.26} and writing $M_n \coloneqq \max_{p_0<j\leq p}|\tilde\alpha_{nj}^{(0)}|$ we obtain 
\begin{align*}
    \left\Vert \tilde\alpha^{(1)}_{n\circ} + \frac{1}{n}(S_{n})_{\circ\circ}\Lambda_{n\circ}T_\circ(\tilde\alpha_n^{(0)})\tilde\alpha_{n\circ}^{(1)} \right\Vert &\geq \left\Vert \frac{1}{n}(S_{n})_{\circ\circ}\Lambda_{n\circ}T_\circ(\tilde\alpha_n^{(0)})\tilde\alpha_{n\circ}^{(1)} \right\Vert - \Vert \tilde\alpha^{(1)}_{n\circ} \Vert \\
   &\geq \frac{1}{n}\frac{\min_{p_0<j\leq p}\lambda_{nj}}{K_X}\frac{1}{M_n} \Vert d^{(1)}_{n\circ} \Vert - M_n \Vert d^{(1)}_{n\circ} \Vert\\
   &= \left(\frac{\min_{p_0<j\leq p}\lambda_{nj} - nK_XM_n^2}{nK_X(M_n)} \right) \Vert d^{(1)}_{n\circ} \Vert
\end{align*}
Recall from \eqref{eq7.24} that $\left\Vert \tilde\alpha^{(1)}_{n\circ} + \frac{1}{n}\left(\tilde\Sigma^{(0)}_{n-}\right)^{-1}_{\circ\circ}\Lambda_{n\circ}T_\circ(\tilde\alpha_n^{(0)})\tilde\alpha_{n\circ}^{(1)} \right\Vert = \mathcal{O}_p\left(\frac{1}{\sqrt{n}}\right)$ and note that $M_n = \mathcal{O}_p\left(\frac{1}{\sqrt{n}}\right)$. Then for sufficiently large $n$, we can estimate $\Vert d_{n\circ}^{(1)} \Vert$ as follows:
\begin{align*}
    \Vert d_{n\circ}^{(1)} \Vert &\leq \left(\frac{nK_X(M_n)}{\min_{p_0<j\leq p}\lambda_{nj} - nK_XM_n^2} \right) \left\Vert \tilde\alpha^{(1)}_{n\circ} + \frac{1}{n}\left(\tilde\Sigma^{(0)}_{n-}\right)^{-1}_{\circ\circ}\Lambda_{n\circ}T_\circ(\tilde\alpha_n^{(0)})\tilde\alpha_{n\circ}^{(1)} \right\Vert \\
    & = \mathcal{O}_p\left(\frac{1}{\min_{p_0<j\leq p}\lambda_{nj}}\right)
    \cip 0
\end{align*}
as $\min_{p_0<j\leq p}\lambda_{nj} \to \infty$. 
Finally, for sufficiently large $n$,
\begin{align*}
    \dfrac{\Vert \tilde\alpha^{(1)}_{n\circ} \Vert}{\Vert \tilde\alpha^{(0)}_{n\circ}  \Vert} \leq \dfrac{ \Vert d_{n\circ}^{(1)} \Vert \max_{p_0 \leq j \leq p}|\tilde\alpha^{(0)}_{nj}|}{\Vert \tilde\alpha^{(0)}_{n\circ}  \Vert} \leq \Vert d_{n\circ}^{(1)} \Vert \cip 0. 
\end{align*}
The proof above relied only on the fact that $\tilde\alpha^{(0)}_{n\circ} = \mathcal{O}_p\left(\frac{1}{\sqrt{n}}\right)$. For $k \geq 1$ we repeat the proof and obtain
\begin{align*}
\dfrac{\Vert \tilde\alpha^{(k+1)}_{n\circ} \Vert}{\Vert \tilde\alpha^{(k)}_{n\circ}  \Vert} \cip 0,
\end{align*}
as claimed in \eqref{eq3.12}. 
To show \eqref{eq3.13}, we first observe
\begin{align*}
    \tilde\beta_n^{\ast(1)} = \left(\frac{Z^{\ast T}\Zast }{n} + \gamastn\frac{S(\bastinit)}{n}\right)^{-1}\frac{1}{n}Z^{\ast T}L_n(\tilde\alpha_n^{(1)}),
\end{align*}
which can be rewritten as
\begin{align*}
    \begin{pmatrix}
    \tilde\beta_{n\star}^{\ast(1)} - \bast_{0\star} \\
    \tilde\beta_{n\circ}^{\ast(1)}
    \end{pmatrix}
    + \left(\frac{Z^{\ast T}\Zast }{n} \right)^{-1}\frac{\gamastn}{n} S(\bastinit)
    \begin{pmatrix}
    \tilde\beta_{n\star}^{\ast(1)}\\
    \tilde\beta_{n\circ}^{\ast(1)}
    \end{pmatrix}
    =\mathcal{O}_p\left(n^{-\frac{c}{2}}\right).
\end{align*}
We may bound similarly as before, but as it is not proven that $\tilde\beta_n^{\ast(k)} = \mathcal{O}_p(n^{-1/2})$, we can only appeal to $\tilde\beta_n^{\ast(k)} = \mathcal{O}_p(n^{-c})$ for $c < 1/2$, thus requiring the slightly stronger condition: 
\begin{equation*}
    \frac{\min_{q_0<j\leq q+1}\gamma_{nj}}{n^{\delta}} \to \infty 
\end{equation*}
for some $\delta \in (0,\frac{1}{2})$. The final step is identical to the case of $\alpha$ so \eqref{eq3.13} is proved. In the case where $\mathbb{E}\left[|\epsilon_1|^{-(1+\tau)}\right] < \infty$, $\tilde\beta_n^{\ast(k)}$ is $\sqrt{n}$-tight for $k \geq 0$, which means the proof follows if $\min_{q_0<j\leq q+1}\gamma_{nj} \to \infty$.


\def\cprime{$'$} \def\polhk#1{\setbox0=\hbox{#1}{\ooalign{\hidewidth \lower1.5ex\hbox{`}\hidewidth\crcr\unhbox0}}} \def\cprime{$'$} \def\cprime{$'$}

\end{document}